\documentclass[12pt]{article}
\usepackage[latin1]{inputenc}
\usepackage[english]{babel}
\usepackage{amsmath}
\usepackage{amsthm}
\usepackage{amsfonts}
\usepackage{amssymb}
\usepackage{enumitem}
\usepackage{graphicx}
\usepackage[hmargin=2cm,paper=a4paper]{geometry}
\usepackage{mathrsfs}
\usepackage{xcolor}

\usepackage[square,numbers]{natbib}
\author{Jean-Dominique Deuschel\footnote{Technische Universit\"at Berlin, Strasse des 17. Juni 136, 10623 Berlin, Germany {\tt deuschel@math.tu-berlin.de}}, 
Henri Elad Altman\footnote{Universit\'{e} Sorbonne Paris Nord, LAGA, CNRS UMR 7539, Institut Galil\'ee, 99 Av. J.-B.
Cl\'ement, 93430 Villetaneuse, France. {\tt elad-altman@math.univ-paris13.fr}}}
\title{Scaling Limits of Line Models in Degenerate Environment}
\date{}

\theoremstyle{plain}
\newtheorem{thm}{Theorem}[section]
\newtheorem{theorem}[thm]{Theorem}
\newtheorem{lemma}[thm]{Lemma}
\newtheorem{proposition}[thm]{Proposition}
\newtheorem{corollary}[thm]{Corollary}

\newtheorem{assumption}[thm]{Assumption}
\theoremstyle{definition}
\newtheorem{definition}[thm]{Definition}

\newtheorem{remark}[thm]{Remark}
\numberwithin{equation}{section}

\begin{document}
\maketitle


\abstract{We consider a 2-dimensional model of random walk in random environment known as line model. The environment is described by two independent families of i.i.d. random variables dictating rates of jumps in vertical, respectively horizontal directions, and whose values are constant along vertical, respect. horizontal lines. When jump rates are heavy-tailed in one of the directions, and given by a stable distribution, we prove that the random walk becomes superdiffusive in that direction, with an explicit scaling limit written as a two-dimensional Brownian motion time-changed (in one of the components) by a Kesten-Spitzer process as introduced in \cite{kesten1979limit}. In the case of a fully degenerate environment, a sufficient condition for non-explosion is provided, and conjectures on the associated scaling limit are formulated.}

\section{Introduction}

\subsection{Outline of the article}

This paper deals with a time-continuous random walk $(X(t))_{t \geq 0}$ on $\mathbb{Z}^2$ in the random environment $\omega$. The walk to nearest neighbors is generated by the (random) operator $\mathcal{L}^\omega$ acting on test functions $f: \mathbb{Z}^2 \to \mathbb{R}$ as
$$\mathcal{L}^\omega f(x) = \sum_{e : |e|=1} \omega(x,x+e) \left(f(x+e) - f(x)\right), $$
with symmetric conductances 
$$0<\omega(x,y)=\omega(y,x)<\infty.$$
The law $\mathbb{P}$ of the environment  is assumed to be stationary and ergodic with respect to the shifts $\{\theta_x\}$ on $\mathbb{Z}^2$.
For fixed environment $\omega$, we denote by $P^\omega_0$ the quenched law of the walk with $P^\omega_0(X(0)=0)=1.$
We are concerned with the question whether, for $\mathbb{P}$ a.e. $\omega$, under $P^\omega_0$,  no explosion occurs and, whether the quenched invariance principle (QIP) holds, that is the diffusively rescaled walk 
$$X^{(T)}(t)\equiv T^{-1/2} X(Tt), \qquad t \geq 0,$$
converges in law as $T \to \infty$ to a non degenerate Brownian motion.

In the past few years, considerable progress has been made in weakening the assumptions on environments ensuring a QIP for the associated random walk. Thus, in \cite{Andres2015invariance}, QIP is obtained for random conductance model on $\mathbb{Z}^d$, $d \geq 2$, for environments that are degenerate (i.e. non elliptic) but are ergodic, and satisfy appropriate moment conditions. The more recent work \cite{BISKUP} showed that, when $d =2$, only 1st moment condition for $\omega$ and $1/\omega$, namely
$$\mathbb{E}[\omega(x,y)]<\infty,\quad \mathbb{E}[1/\omega(x,y)]<\infty,$$
 is sufficient for a QIP to hold. It should be noted, that in case of i.i.d. conductances defined on $\mathbb{Z}^d$, with $d \geq 2$, no such moment condition is required for the QIP, cf.  \cite{ABDH},
however for general ergodic $\mathbb{P}$ the moment condition is relevant for both non explosion and QIP. Generalizations of the results of \cite{Andres2015invariance} to environments that are non-symmetric have been obtained in the recent article \cite{deuschel2024quenchedinvarianceprinciplerandom}. We also refer the reader to \cite[Chap. 3]{komorowski2012fluctuations}, which establishes central limit theorems in $L^1$ with respect to the environment.
{

In contrast to the aforementioned works, this paper explores environments for which we \textbf{do not} observe a QIP. More specifically, we consider a random conductance model on $\mathbb{Z}^2$, with conductances that are constant along horizontal and vertical lines, and given by
\begin{align*}
  \omega(x,x +e_1) &= \omega(x,x -e_1) = H(x_2)(\omega) > 0, \\
  \omega(x,x +e_2) &= \omega(x,x -e_2) = V(x_1)(\omega) > 0,
\end{align*}
for $x=(x_1,x_2) \in \mathbb{Z}^2$, where $e_1=(1,0)$ and $e_2=(0,1)$, and where $(H(k))_{k \in \mathbb{Z}}$ and $(V(k))_{k \in \mathbb{Z}}$ are two independent families of i.i.d. positive random variables. Of course this implies the ergodicity of $\mathbb{P}$, which is however not directional ergodic.

Let us denote by $X(t) = (X_1(t),X_2(t))$, $t \geq 0$, the random walk with rates as above. Constant conductances along lines imply that  both components $X_1$ and $X_2$ have no drift. As a consequence, introducing for all $R>0$ the stopping time 
$$\tau_R := \inf \{t \geq 0: |X_1(t)| + |X_2(t)| \geq R\},$$ 
then the stopped processes 
\[X_i^{\tau_R}(t) := X_i(t\wedge \tau_R), \qquad i =1,2,\] 
are martingales with predictable quadratic variation 
$$<X^{\tau_R}_1>(t) = \int_0^{t \wedge \tau_R} 2H(X_2(s))\,ds, \qquad <X^{\tau_R}_2>(t) = \int_0^{t \wedge \tau_R} 2V(X_1(s))\,ds,$$
while we have $<X^{\tau_R}_1,X^{\tau_R}_2>(t)=0$ for all $t \geq 0$. The additive functionals above can be interpreted as random walks in random scenery, although here the random walk is dependent of the ambient scenery.

Regarding the jump rates, we assume $\{H(x_2), x_2\in \mathbb{Z}\}$ and $\{V(x_1), x_1\in \mathbb{Z}\}$ to be two independent families of i.i.d. positive random variables, with tails given, as $L \to \infty$, by
\[ \mathbb{P}(H(0)>L) = c_H L^{-\alpha_1(1 + o(1))}, \quad \mathbb{P}(V(0)>L) = c_V L^{-\alpha_2(1 + o(1))},  \]
 for some exponents $\alpha_i>0$ and some constants $c_H, c_V >0$. In the statement of our scaling limit results in section \ref{sec:scaling} we will in fact define the random variables $H(x_2)$, $x_2 \in \mathbb{Z}$, in terms of an underlying stable process, see Assumption \ref{assump:H_V_tail_bis}.
  
  In case $\alpha_1,\alpha_2>1$, we have $\mathbb{E}[H(x_2)]<\infty$ and $\mathbb{E}[V(x_1)]<\infty$, then no explosion occurs and the QIP holds.
  
  In this paper we assume either 
  $$\alpha_1<1\text{ and  }\alpha_2<1, \qquad \hbox{so that} \qquad \mathbb{E}[H(x_2)]=\infty, \, \, \mathbb{E}[V(x_1)]=\infty, \qquad \text{(case 1)}$$
   or
   $$\alpha_1<1\text{ and  }\alpha_2>1,\qquad \hbox{so that} \qquad \mathbb{E}[H(x_2)]=\infty, \, \, \mathbb{E}[V(x_1)]<\infty,\qquad \text{(case 2)}. $$

  The case where the random variables $V(x_1)$ are a deterministic constant was the object of 2 joint papers \cite{DF1} and \cite{DF2}, where heat kernel estimates were derived.
  
  In neither cases 1,2 is it a priori clear whether non explosion occurs. Also the random walk might travel very far  along lines with very large conductances, thus we expect
  a super diffusive behaviour, i.e. a singular scaling.
  
   In order to identify the correct scaling let us suppose that, as above 
   $X(t)=(X_1(t),X_2(t))$. We assume that $X_1$ and $X_2$ do not explode in finite time, so they are local martingales with quadratic variation $2A_1^H$ and $2A_2^V$, where
   \begin{equation}
   \label{eq:qv.a1a1.intro}
   A_1^H(t) := \int_0^{t} 2H(X_2(s))\,ds, \quad A_2^V(t) := \int_0^{t} 2V(X_1(s))\,ds, \qquad t \geq 0.  
   \end{equation}
  Moreover assume that, for some $\gamma_i\ge 1/2$,
  \begin{equation}
  \label{eq:rescaled.process.case1}
  X^{(T)}(t)=(T^{-\gamma_1}X_1(T t), T^{-\gamma_2} X_2(Tt)),
  \end{equation}
  converges in law, as $T \to \infty$, to some limiting process $Z(t)=(Z_1(t),Z_2(t))$. Then with
 $$\delta_1=1-{\gamma_2}+{\gamma_2}/{\alpha_1},\qquad \delta_2=1-{\gamma_1}+{\gamma_1}/{\alpha_2},$$
we can heuristically apply Kesten-Spitzer's scaling limit results of \cite{kesten1979limit} for random walks in random scenery (originally written for discrete time $1/\gamma_i$ stable random walks). We expect that, as $T \to \infty$, 
 $$(T^{-\delta_1}A^H_1(T t), T^{-\delta_2} A^V_2(T t)), \qquad t \geq 0,$$
 converges to a couple of monotone increasing  self similar processes $(\Delta_1(t),\Delta_2(t))_{t \geq 0}$.
 Computing the corresponding quadratic variations
 $$(<X_1^{(T)}>(t),<X_2^{(T)}>(t))=2(T^{-2\gamma_1}A^H_1(T t),T^{-2\gamma_2}A^V_2(T t)),$$
 in view of the above convergence, we get the following equality
 $$2\gamma_1=\delta_1=1-\gamma_2+\frac{\gamma_2}{\alpha_1},\qquad2\gamma_2=\delta_2=1-\gamma_1+\frac{\gamma_1}{\alpha_2}.$$
 This is a simple 2-d linear equation: set 
 $$\epsilon_1=\frac{{1}/{\alpha_1}-1}{2}\vee 0,\qquad\epsilon_2=\frac{{1}/{\alpha_2}-1}{2}\vee 0, $$
 then, provided 
 $$\epsilon_1\cdot \epsilon_2<1,$$
  (needed in order to get $\gamma_i>0$), we get
 $$\gamma_1=\frac{1+\epsilon_1}{2(1-\epsilon_1\cdot\epsilon_2)}\ge 1/2, \qquad \gamma_2=\frac{1+\epsilon_2}{2(1-\epsilon_1\cdot\epsilon_2)}\ge 1/2.$$
It should be noted that the processes $X_1$ and $X_2$ are dependent, thus the additive functionals of \eqref{eq:qv.a1a1.intro} above do not fall within the scope of \cite{kesten1979limit}, in particular we cannot expect convergence of the local times of $X_1$ and $X_2$! Thus the above derivation of the scaling exponents $\gamma_1$ and $\gamma_2$ is non-rigorous, however it turns out to provide a good guess for conditions of non-explosion of our random walk.

Thus, our first result, Proposition \ref{prop:X_is_conservative}, shows that the condition 
  $$\epsilon_1\cdot \epsilon_2<1$$
 implies that for $\mathbb{P}$ a.e. $\omega$, under $P^\omega_0$, no explosion occurs. 
 Since 
 $$\lim_{\epsilon_1\cdot\epsilon_2\uparrow 1}\,\gamma_i\ge\lim_{\epsilon_1\cdot\epsilon_2\uparrow 1}\,\frac{1}{1-\epsilon_1\epsilon_2}=\infty,$$ 
 we expect that when $\epsilon_1\cdot \epsilon_2>1$ explosion occurs. Moreover with $\gamma_i^+>\gamma_i\ge 1/2$ and 
 \[\begin{split}
 \label{eq:rescaled.process.case1.supercit}
 X^{T,+}(t)=(T^{-\gamma^+_1}X_1(T t),T^{-\gamma^+_2}X_2(Tt)),
 \end{split}\]
 we show that  for $\mathbb{P}$ a.e. $\omega$, for all $t>0$, the convergence
 $$\lim_{T \to \infty} \sup_{0\le s \le t} |X_i^{T,+}(s)|=0, \qquad i =1,2,$$
holds in $P^\omega_0$-probability, see Proposition \ref{prop:over_scaling} below.

In case 2 we can be more precise: in that case since $\epsilon_2=0$,  no explosion occurs for all $\alpha_1>0$.
Moreover, assuming that the random variables $H(x)$, $x \in \mathbb{Z}$, follow a stable distribution of index $\alpha_1$, 
then setting 
$$\gamma_1= \frac{1}{2}\left(\frac{1}{2 \alpha_1} + \frac{1}{2}\right)>1/2,\qquad \gamma_2=1/2,$$ 
and performing an appropriate rescaling of the environment together with the random walk, then for $\mathbb{P}$-a.e. realisation of the environment, as $T \to \infty$, the process 
   $$X^{(T)}(t)=(T^{-\gamma_1}X_1(Tt),T^{-1/2} X_2(Tt)), \qquad t \geq 0,$$
converges in law, and for the Skorohod topology, to the process 
$$(B_1\circ \Delta^{\mathcal{H}} (B_2),B_2)$$
where $B_1, B_2$ are two independent Brownian motions and
\begin{equation}
\label{eq:def.delta_intro}
\Delta^{\mathcal{H}} (B_2) (t) =\int_{\mathbb{R}} L^{B_2}(t,x) d\mathcal{H}(x), \qquad t \geq 0.
\end{equation}
Here $(\mathcal{H}(x))_{x \in \mathbb{R}}$ is a two-sided stable subordinator of index $\alpha_1$ and, for all $(t,x) \in \mathbb{R}_+ \times \mathbb{R}$, $L^{B_2}(t,x)$ is the local time of $B_2$, up to time $t$, at the point $x$. 
See Theorem \ref{thm:conv} below for the precise statement. 

Processes such as $\Delta^{\mathcal{H}} (B_2)$ appear in the work of Kesten and Spitzer \cite{kesten1979limit} to describe the scaling limits of the models of random walk in random scenery considered therein. 
Since in case 2, $\mathbb{E}[V(x_1)]<\infty$, the  convergence of $X_2^{(T)} (t)\equiv T^{-1/2} X_2(Tt)$ to $B_2(t)$ follows from a martingale scaling limit theorem, wherefrom the results of \cite{kesten1979limit} heuristically entail the convergence of the first additive functional $A^H_1(t)$ of \eqref{eq:qv.a1a1.intro}, when properly rescaled, to \eqref{eq:def.delta_intro}. 
However, in contrast to \cite{kesten1979limit} a major caveat arises in our setting, which is the lack of independence of the process $(X_2(t))_{t \geq 0}$ with the random variables $H(k)$, $k \in \mathbb{Z}$. In addition, the convergence of the local times of $X_2^{(T)}$ towards those of $B_2$ is not clear at all. To overcome these difficulties, we resort to an appropriate time-change, which transforms the diffusive component $X_2$ into a \textit{simple random walk}. Undoing the time change hinges on an appropriate application of Birkhoff's ergodic theorem for the environment seen from the walker and relies on the random variables $V(k)$, $k \in \mathbb{Z}$, being integrable. 

In case 1 we expect that $X^{(T)}$ converges in law, for the Skorohod topology, to a process $(Z_1,Z_2)$ of the form
\begin{equation}
\label{eq:limit.process.case1}
(Z_1(t),Z_2(t))=(B_1\circ \Delta^\mathcal{H}(Z_2)(t),B_2\circ \Delta^\mathcal{V}(Z_1)(t)), \qquad t \geq 0,
\end{equation}
where
\begin{equation}
\label{eq:def.ahav.limit}
( \Delta^\mathcal{H}(Z_2)(t),  \Delta^\mathcal{V}(Z_1)(t)) := \left(\int_{\mathbb{R}} L^{Z_2}(t,x) \, \mathcal{H} (dx), \int_{\mathbb{R}} L^{Z_1}(t,x) \ \mathcal{V}(dx)\right),
 \end{equation}
where $L^{Z_i}(t,x)$ are the occupation densities of $Z_i, i \in \{1,2\}$, and $(\mathcal{H}(x))_{x \in \mathbb{R}}, (\mathcal{V}(x))_{x \in \mathbb{R}}$ are independent, double-sided stable subordinators of respective indices $\alpha_1, \alpha_2$, that are independent of $B_1,B_2$. Formally, such a process would correspond to a solution of the system of singular SDEs
\[\begin{split}
Z_1(t) = \int_0^t \sqrt{\dot{\mathcal{H}}(Z_2(s))} \, dW_1(s) \\
Z_2(t) = \int_0^t \sqrt{\dot{\mathcal{V}}(Z_1(s))} \, dW_2(s) 
\end{split}\]
where $W_1, W_2$ are two independent Brownian motions. These SDEs are very singular, as $\dot{\mathcal{H}}$ and $\dot{\mathcal{V}}$ are only distributions, thus the very existence of such a process $(Z_1,Z_2)$ is highly non-trivial. The corresponding Dirichlet form on $L^2(\mathbb{R}^2,dx)$ is associated with the
bilinear form
$$\mathcal{E}(u,v)=\int_{\mathbb{R}^2} \partial_{x_1}u(x) \partial_{x_1}v(x) d \mathcal{H}(x_2) dx_1
+ \partial_{x_2}u(x) \partial_{x_2}v(x) d \mathcal{V} (x_1) dx_2,\quad
u,v\in C^1_c(\mathbb{R}^2),$$
cf. \cite{CF}, \cite{FO}. We also fall short of proving tightness of the rescaled process \eqref{eq:rescaled.process.case1}, as our result for the convergence of \eqref{eq:rescaled.process.case1.supercit} (see Proposition \ref{prop:over_scaling}) based on quenched moment estimates is too rough to set $\gamma^+_i=\gamma_i$ in the scaling.


Besides the (variable-speed) line model random walk $(X(t))_{t \geq 0}$ considered above, we also consider in this work the associated constant speed random walk $(\hat X(t))_{t \geq 0}$ generated by 
 $$\hat{\mathcal{L}}^\omega f(x) = \sum_{e : |e|=1}\hat\omega(x,x+e) \left(f(x+e) - f(x)\right), $$
where
$$\hat\omega(x,x\pm e_1)=\frac{H(x_2)}{H(x_2)+V(x_1)},\qquad \hat\omega(x,x\pm e_2)=\frac{V(x_1)}{H(x_2)+V(x_1)},$$
which corresponds to a simple time change of the original walk $X$, namely we can construct $\hat{X}$ as
$\hat X(t)=X\circ {J^{H,V}}^{-1}(t)$, where
$${J^{H,V}}^{-1}(t)=\inf\{s>0: J^{H,V}(s) \ge t\}, \qquad t \geq 0, $$
is the inverse of the increasing function $J^{H,V}(s) := A^H_1(s)+A^V_2(s)$, $s \geq 0$ (see \eqref{eq:qv.a1a1.intro}).
\newline
In case 2, assuming again the random variables $H(k)$, $k \in \mathbb{Z}$, to follow an $\alpha_1$-stable distribution, then the random walk is trapped along the vertical direction and $\hat X_2$ sub-diffusively rescaled converges to a FIN diffusion as considered in \cite{FIN} and \cite{benarous2005},  whereas in the horizontal direction $\hat X_1$ rescales diffusively. More precisely
with $\gamma_1=\frac{1}{2} (\frac{1}{2 \alpha_1}+\frac{1}{2})$, (so that we have $1/(4\gamma_1)<1/2$), set
$$\hat X^{(T)}(t)=(T^{-1/2}\hat X_1(t/\epsilon),T^{-1/(4\gamma_1)} \hat X_2(T t)), \qquad t \geq 0.$$
Then, performing an appropriate rescaling of the environment together with the random walk, for $\mathbb{P}$ a.e. $\omega$, the process $(\hat X^{(T)}(t))_{t \geq 0}$ converges in law for the Skorohod topology, as $T \to \infty$, to $(B_1(t),B_2\circ \Delta^\mathcal{H}(B_2)^{-1}(t))$ where  $B_1, B_2$ are two independent Brownian motions and
$$ \Delta^\mathcal{H}(B_2)^{-1}(t)=\inf\{s>0: \Delta^\mathcal{H}(B_2)(s) \ge t\},$$
is the inverse of the process $\Delta^\mathcal{H}(B_2)(s) =\int_{\mathbb{R}}L^{B_2}(s,x)d\mathcal{H}(x)$, $s \ge 0$, with $(\mathcal{H}(x))_{x \in \mathbb{R}}$ a two-sided stable subordinator of index $\alpha_1$, see Theorem \ref{thm:limit_csrw} below.

 As for the variable-speed random walk, the scaling limit of $(\hat{X}(t))_{t \geq 0}$ in case 1 remains open but assuming as above that $X^{(T)}(t)=(T^{-\gamma_1}X_1(Tt),T^{-\gamma_2}X_2(Tt)),$
 converges to $(Z_1(t),Z_2(t))$ of the form \eqref{eq:limit.process.case1}, then if $\alpha_1<\alpha_2<1$, i.e. $\gamma_1<\gamma_2$,  we expect that
 $$\hat X^{(T)}(t)=(T^{-1/2} \hat{X}_1(Tt),T^{-\gamma_2/(2\gamma_1)}\hat X_2(Tt)), \qquad t \geq 0,$$
 converges as $T \to \infty$ to $(B_1(t),B_2\circ(\Delta^\mathcal{V}(Z_1)\circ \Delta^\mathcal{H}(Z_2)^{-1}(t)))$, where $\Delta^\mathcal{V}(Z_1)$ and $\Delta^\mathcal{H}(Z_2)$ are as in \eqref{eq:def.ahav.limit} and
 \[ \Delta^\mathcal{H}(Z_2)(t) = \inf\{s>0: \Delta^\mathcal{H}(Z_2)(s) \ge t\}, \qquad t \geq 0.  \]
 
\subsection{Outline of the article}

In section 2 we first  introduce the line model in detail and prove a few basic properties thereof. Section 3 provides a non-explosion criterion (Proposition \ref{prop:X_is_conservative}) which we believe to be optimal. Then Section 4 proves, in the semi-degenerate case where $\alpha_1<1$ but $\alpha_2 > 1$, i.e. in case 2, scaling limit results for the line model (Theorem \ref{thm:conv}) as well as for the constant-speed variant thereof (Theorem \ref{thm:limit_csrw}). Section 5 is an appendix containing a few definitions and technical results needed in our proofs. 
 \subsection{Notations}
 
Throughout the article, for all $x \in \mathbb{Z}^2$, we shall denote by $\|x\| := |x_1| + |x_2|$ the $1$-norm of $x$. We will denote by $C := C([0,\infty))$ the space of continuous, real-valued functions on $[0,+\infty)$, which we shall equip with the topology of uniform convergence on finite intervals. We shall denote by $D=D([0,\infty))$ the space of c\`{a}dl\`{a}g functions on $[0,+\infty)$, equipped with the Skorohod topology, see Section 16 of \cite{billingsleySndEdition} for the definition.

\paragraph{Acknowledgements:} This project started in collaboration with \textbf{Toyomu Matsuda}, then PhD student at Freie Universit\"{a}t Berlin. This work benefitted from insightful discussions with Nicolas Perkowski. We would like to thank Zhen-Qing Chen for valuable advices on the skew product of diffusion processes. The first author's research has been supported by TRR 388. At the onset of this project the second author was employed as Dirichlet Postdoctoral Fellow in Mathematics at Freie Universit\"{a}t Berlin, with funding provided by MATH+, in the framework of the ``MATH+ EXC 2046'' research project.
 
 \begin{figure}[ht]
        \centering
        \includegraphics[width= 0.6\textwidth]{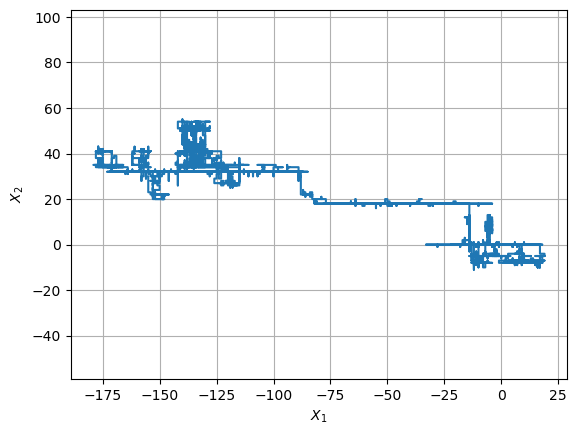}
        \caption{Trajectory of line model for $\alpha_1=0.6$, $\alpha_2=0.9$, after 100 jumps}
\end{figure}
 

\section{The model}
\label{sec:model}
Let $(\mathbb{Z}^2, \mathbb{E}^2)$ be the infinite $2$-dimensional Euclidean lattice, where $\mathbb{E}^2$ denotes the set of edges between any two neigboring points of $\mathbb{Z}^2$:
$$\mathbb{E}^2 =  \{(x,y), x, y \in \mathbb{Z}^2 \, \hbox{s.t.} \, ||x-y||=1\}.$$
The set of environments on this graph is given by
\begin{equation}
\label{def:Omega}
\Omega:= \{\omega: \mathbb{E}^2 \to \mathbb{R}\}
\end{equation}
endowed with the cylindrical $\sigma$-algebra $\mathcal{F}$. We consider two independent families of i.i.d. positive random variables $(H(k))_{k \in \mathbb{Z}}$ and 
$(V(k))_{k \in \mathbb{Z}}$. We define a random environment $\omega$ by setting, for any point $x=(x_1,x_2) \in \mathbb{Z}^2$,
\begin{align*}
  \omega(x,x +e_1) &= \omega(x,x -e_1) = H(x_2)(\omega) > 0, \\
  \omega(x,x +e_2) &= \omega(x,x -e_2) = V(x_1)(\omega) > 0,
\end{align*}
where $e_1=(1,0)$ and $e_2=(0,1)$. 
If $|x - y| \neq 1$, we set $\omega(x, y) := 0$. 
Observe that 
\begin{equation}
\label{eq:random_cond}
\forall x,y \in \mathbb{Z}^2, \qquad \omega(x, y) = \omega(y, x). 
\end{equation}
We denote by $\mathbb{P}$ the probability law thus induced on  $(\Omega, \mathcal{F})$. For all $\omega \in \Omega$, we then introduce a continuous-time random walk $(X(t))_{t \geq 0}$ on $\mathbb{Z}^2$ with generator 
$$\mathcal{L}^\omega f(x) = \sum_{e : |e|=1} \omega(x,x+e) \left(f(x+e) - f(x)\right), $$
and denote by $P^\omega_0$ the law of this walk started from $0$. In particular, due to \eqref{eq:random_cond}, $\mathcal{L}^{\omega}$ is symmetric with respect to the counting measure. This is thus an instance of a random conductance model. Finally we denote by $\mathbf{P} := \int \mathbb{P} (d \omega) \, P^\omega_0$ the annealed probability measure.
\newline
From now on and and until Section 4.1 included, we make the following assumption regarding the environment:
\begin{assumption}\label{assump:H_V_tail}
  We assume:
  \begin{enumerate}
  \item \textbf{Hypothesis (H1)}: there exists $c>0$ such that $H(x) \geq c$ and $V(x) \geq c$ for all $x \in \mathbb{Z}$.
  \item \textbf{Hypothesis (H2)}: there exist exponents $\alpha_1, \alpha_2 > 0$ and constants $c_H, c_V > 0$ such that, when $L \to \infty$,
  \[ \mathbb{P}(H(x_2)>L) = c_H L^{-\alpha_1(1 + o(1))}, \quad \mathbb{P}(V(x_1)>L) = c_V L^{-\alpha_2(1 + o(1))}.  \]
  \end{enumerate}
\end{assumption}

\subsection{Ergodicity of the environment}

The following property will be useful:

\begin{lemma}
    The environment is stationary and ergodic. 
\end{lemma}

\begin{proof}
    The sequences $V(x_1)$, $x_1 \in \mathbb{Z}$, and $H(x_2)$, $x_2 \in \mathbb{Z}$, are both i.i.d., hence in particular stationary. 
    Since they are also mutually independent, the stationarity of the environment follows. 
    The proof of ergodicity is standard, noting that for any two functions $f,g \in L^2(\Omega)$ that are local, i.e. that depend only on values of configurations on edges in some finite set $K$, we have
     \[\mathbb{E}[f \tau_{k(e_1+e_2)}g] = \mathbb{E}[f] \, \mathbb{E}[g],  \]
     as soon as $k> 2 \, \hbox{diam}(K)$. Here, for $k \geq 1$ we define $\tau_{k (e_1+e_2)} g \in L^2(\Omega)$ by
    \[(\tau_{k (e_1+e_2)} g)(\omega) := g(\tau_{k (e_1+e_2)} \omega), \qquad \omega \in \Omega.\] 
    See e.g. the reasoning in the 2nd step of \cite[Thm 5.4]{Andres2013HarnackIO} for more details. 
    
\end{proof}

We introduce the process known as \textit{the environment seen from the walker}. Let $\Omega$ be, as in \eqref{def:Omega}, the set of environments, and define, for all $\omega \in \Omega$, the process
\begin{equation}
\label{eq;env_seen_from_partic}
\omega(t) := \tau_{X_t} \omega, \qquad t \geq 0,
\end{equation}
where, for all $\omega \in \Omega$ and $u \in \mathbb{Z}^2$, $\tau_{u} \omega$ denotes the environment $\omega$ shifted by $u$:
$$\forall e \in \mathbb{E}^2, \quad (\tau_u \omega)_e := \omega_{e-u}. $$
Under the annealed law $\mathbf{P} = \int \mathbb{P} (d\omega) \, P^\omega_0$, this defines a Markov process on $\Omega$, with reversible measure $\mathbb{P}$. By ergodicity of the environment, this process is furthermore ergodic, see e.g. \cite[Corollary 2.1.25]{Zeitouni:2004aa}.  


\subsection{Martingale property of the random walk}

Another property that plays a key role in this model is the martingale property of the random walk. We introduce the additive functionals
\[A^H_1 (t) := \int_0^t H(X_2(s)) \, ds, \qquad  A^V_2(t) := \int_0^t V(X_1(s)) \, ds.\]
Let $(\mathcal{F}_t)_{t \geq 0}$ denote the natural filtration: 
\[\mathcal{F}_t := \sigma(\{X_s, s \leq t\}), \qquad t \geq 0,\]
and, for all $R>0$, let $\tau_R : = \inf \{t \geq 0: \|X(t)\| \geq R\}$. Note that $\tau_R$ is a stopping time with respect to $(\mathcal{F}_t)_{t \geq 0}$. 

\begin{lemma}
For all $\omega \in \Omega$, for all $R>0$, the stopped processes 
\[X_i^{\tau_R}(t) := X_i(t\wedge \tau_R), \qquad i =1,2\] 
are martingales for the quenched probability measure $P^\omega_0$ with predictable quadratic variations 
\[<X^{\tau_R}_1>(t)=2A^H_1(t \wedge \tau_R), \quad <X^{\tau_R}_2>(t)=2 A^V_2(t \wedge \tau_R), \qquad t \geq 0. \]
\end{lemma}

\begin{proof}
To see these claims, let  $\varphi_R: \mathbb{Z}^2 \to \mathbb{R}$ be the bounded function given by $\varphi_R(x) := \mathbf{1}_{\{\|x\| \leq R\}} \, x_1$, $x \in \mathbb{Z}^2$. By Dynkin's formula (see e.g. Theorem 3.32 in \cite{Liggett2010ContinuousTM}), we have
\[\varphi_R(X(t)) - \varphi_R(0) - \int_0^t \mathcal{L}^{\omega} \varphi_R(X(s)) \, ds = M_t, \]
where $(M_t)_{t \geq 0}$ is a c\`{a}dl\`{a}g martingale. By the optional stopping theorem, and using the fact that $\varphi_R(X_t) = X_1(t)$ for all $t \leq \tau_R$ and $\mathcal{L}^{\omega} \varphi_R(X_t) = 0$ for all $t < \tau_R$, we get
\[X_1^{\tau_R}(t) = M_{t \wedge \tau_R}, \qquad t \geq 0, \]
which is a martingale.
    The same holds for $X_2$ similarly. 
    Regarding the expression for the quadratic variation, 
    applying Dynkin's formula with the function $\psi_R(x) = \mathbf{1}_{\{\|x\| \leq R\}} \, (x_1)^2$, then stopping at $\tau_R$ and using the fact that $\varphi_R(X_t) = X_1(t)^2$ for all $t \leq \tau_R$ and $\mathcal{L}^{\omega} \varphi_R(X_t) = 2 H(X_2(t))$ for all $t < \tau_R$, we get that
    \[(X_1(t \wedge \tau_R))^2 - 2 \int_0^{t\wedge \tau_R}   H(X_2(s)) \, \mathrm{d} s\]
     is a martingale, yielding the expression of $\langle X^{\tau_R}_1, X^{\tau_R}_1 \rangle$. The expression for $\langle X^{\tau_R}_2, X^{\tau_R}_2 \rangle$ is obtained similarly.
\end{proof}

\section{Non-explosion}

Before studying the long-time behaviour of the random walk $X$ defined in Section \ref{sec:model}, we need to ensure that this random walk does not blow up 
in finite time for almost all environments. For this purpose, as in the introduction we denote 
\begin{equation}
\label{def:epsilon}
\varepsilon_i := \frac{1}{2} \left(\frac{1}{\alpha_i} -1\right) \wedge 0, \qquad i=1,2. 
\end{equation}
\begin{proposition}\label{prop:X_is_conservative}
  Assume $\varepsilon_1 \varepsilon_2 < 1$. Then, for almost all environments $\omega$, the random walk $X$ is conservative, i.e., for every $t < \infty$ 
  we have 
  \begin{align*}
    P^{\omega}_0(\text{the number of jumps of $X$ in $[0, t]$ is finite}) = 1.
  \end{align*}
\end{proposition}

\begin{remark}
In the framework of random conductance models, a standard tool for obtaining non-explosion is provided by considering the chemical distance $d_{\omega}$ defined as 
  \begin{align*}
    d_{\omega}(x, y) := \inf_{\gamma} \Big\{ \sum_i 1 \wedge \omega(z_i, z_{i+1})^{-1/2} \Big\},
  \end{align*}
  where the infimum is taken over all finite paths $\gamma = (z_0, \ldots, z_l)$ connecting two points $x$ and $y$ of $\mathbb{Z}^2$. 
  According to \cite[Lemma~2.5]{Barlow:2010aa}, it suffices to prove 
  \begin{align}\label{eq:chem_dist_exp_sum}
    \sum_{x \in \mathbb{Z}^2} e^{-d_{\omega}(0, x)} < \infty
  \end{align}
  for almost all $\omega$. With this criterion, by getting a lower bound on $d_{\omega}$ along the lines of \cite[Lemma~1.12]{Andres:2016aa}, one can obtain non-explosion provided
  $\alpha_1 \alpha_2 > \frac{1}{4}$. This sufficient condition is however strictly more restrictive than $\varepsilon_1 \varepsilon_2 < 1$. Actually, a time-change as done in the proof below seems 
  necessary to obtain the latter criterion. 
\end{remark}

To prove Proposition \ref{prop:X_is_conservative}, we will first introduce an auxiliary process, which is a time-change of the process $(X(t))_{t \geq 0}$. To do so let us henceforth introduce  $\alpha_1^{-}, \alpha_2^{-}$ two positive constants such that $\alpha_1^{-} < \alpha_1$ and $\alpha_2^{-} < \alpha_2$. In consistency with \eqref{def:epsilon} we correspondingly define
\begin{equation}
\label{def:epsilon+}
\varepsilon_i^+ := \frac{1}{2} \left(\frac{1}{\alpha_i^{-}}-1 \right) \wedge 0, \qquad i =1,2.
\end{equation}
We consider the random walk $(X^\ast_t)_{t \geq 0}$ with generator
$$\mathcal{L}^{\omega^\ast} f(x) = \sum_{e : |e|=1} \omega^\ast (x,x+e) \left(f(x+e) - f(x)\right), $$
where 
$$\forall x,y \in \mathbb{Z}^2, \quad \omega^\ast(x,y) = \frac{\omega(x,y)}{H(x_2)^{\alpha_1^{-}} \, V(x_1)^{\alpha_2^{-}}},$$
that is, for all $x \in \mathbb{Z}^2$
$$\omega^\ast(x,x \pm e_1) = H(x_2)^{1-\alpha_1^{-}} \, V(x_1)^{-\alpha_2^{-}}, \qquad \omega^\ast(x,x \pm e_2) = H(x_2)^{-\alpha_1^{-}} \, V(x_1)^{1-\alpha_2^{-}}, $$
and $\omega^\ast(x,y)=0$ if $|x-y| \neq 1$.

We will prove the following lemma

\begin{lemma}\label{lem:Xstar_is_conservative}
  Assume $\varepsilon_1 \varepsilon_2 < 1$. Then, for almost all environments $\omega$, the random walk $X^\ast$ is conservative, i.e., for every $t < \infty$ 
  we have 
  \begin{align*}
    P^{\omega}(\text{the number of jumps of $X^\ast$ in $[0, t]$ is finite}) = 1.
  \end{align*}
  \end{lemma}

The Proposition \ref{prop:X_is_conservative} then follows from the above Lemma. Indeed note (see e.g. exercice 2.32 in \cite{Liggett2010ContinuousTM}) that $X^\ast$ may be constructed using $X$ via the time-change
\begin{equation}
\label{eq:time_change_x_ast_x}
X^\ast_t = X_{A^{-1}_t}, 
\end{equation}
where $A_t := \int_0^t H(X_2(s))^{\alpha_1^{-}} V(X_1(s))^{\alpha_2^{-}} \, ds $ and $A^{-1}$ denotes the right-continuous inverse of  $A$. Now, 
by Hypothesis (H2) in Assumptions \ref{assump:H_V_tail}, for all $x_2 \in \mathbb{Z}$, as $u \rightarrow +\infty$
$$ \mathbb{P}(H(x_2)^{\alpha_1^{-}}>u) = \mathbb{P}(H(x_2) > u^{\frac{1}{\alpha_1^{-}}}) = u^{-\frac{\alpha_1}{\alpha_1^{-}}}(1 + o(1)). $$
Since $\alpha_1 > \alpha_1^{-}$, we deduce that $\mathbb{E}[H(x_2)^{\alpha_1^{-}}] < \infty$. Similarly, for all $x_1 \in \mathbb{Z}$, 
$\mathbb{E}[V(x_1)^{\alpha_2^{-}}] < \infty$. Hence 
\begin{equation}
\label{eq:time_change_admiss}
\mathbb{E}[H(x_2)^{\alpha_1^{-}} \, V(X_1)^{\alpha_2^{-}}] = \mathbb{E}[H(x_2)^{\alpha_1^{-}}] \, \mathbb{E}[V(x_1)^{\alpha_2^{-}}]  < \infty. 
\end{equation}
This ensures that the time-change defined by \eqref{eq:time_change_x_ast_x} is admissible. Hence, if $X^\ast$ is conservative, then $X$ will be as well.
We are thus left with proving Lemma \ref{lem:Xstar_is_conservative}.

\begin{proof}[Proof of Lemma \ref{lem:Xstar_is_conservative}]
We henceforth fix an environment $\omega \in \Omega$. We will actually prove the stronger statement that, for all $t>0$, $X^\ast_1(t)$ and $X^\ast_2(t)$ have finite moments. 
Let $R >0$, and consider 
$$\tau_R := \inf\{t \geq 0: |X^\ast_1(t)| + |X^\ast_2(t)| > R \}.$$ 
which is a stopping time for the canonical filtration $\mathcal{F}_t := \sigma (\{X^\ast(s):  s \leq t \})$. Henceforth we will mostly write $\tau$ instead of $\tau_R$ to alleviate notations.
For all $p \geq 1$, we set 
$$h_p(t) := E^\omega_0 \left[X^\ast_1(t \wedge \tau)^p \right], \qquad v_p(t) := E^\omega_0 \left[X^\ast_2 (t \wedge \tau)^p \right]. $$
\textbf{ First case:} Assume first that $\varepsilon_1 <1$ and $\varepsilon_2<1$. We will show that,
for all $T>0$, 
\[ \forall t \in [0,T], \qquad h_2(t) \leq C(\omega) (T+ T^{A_1}) \quad  \hbox{and} \quad v_2(t) \leq C(\omega) (T+T^{A_2}),\]
where 
\begin{equation}
\label{def:ai}
A_i :=  \frac{1+ \varepsilon_i^+}{1- \varepsilon_1^+ \varepsilon_2^+} > 1, \qquad i=1,2. 
\end{equation} 
and where henceforth $C(\omega)$ denotes a positive constant depending on $\omega$ but independent of $t$ and of the trajectory $X^\ast$ of the random  walk, and whose values may change from line to line. The process 
$$M_t := (X^\ast_1(t \wedge \tau))^2 - \int_0^{t \wedge \tau} H(X^\ast_2(s))^{1-\alpha_1^{-}} \, V(X^\ast_1(s))^{-\alpha_2^{-}} \, ds$$ 
is a martingale bounded - thanks to the stopping time $\tau$ - by 
\[R^2 + t \times \max_{\|x\| \leq R} H(x_1)^{1-\alpha_1^{-}} \, V(x_2)^{-\alpha_2^{-}} < \infty. \]
In addition $M_0=0$. As a consequence 
\[\begin{split}
E^\omega_0[(X^\ast_1(t \wedge \tau))^2]  &=  E^\omega_0 \left[\int_0^{t\wedge \tau}  H(X^\ast_2(s))^{1-\alpha_1^{-}} \, V(X^\ast_1(s))^{-\alpha_2^{-}} \, ds \right] \\
&\leq E^\omega_0 \left[\int_0^{t}  H(X^\ast_2(s \wedge \tau))^{1-\alpha_1^{-}} \, V(X^\ast_1(s \wedge \tau))^{-\alpha_2^{-}} \, ds \right].
\end{split}\]
To bound the integrand in the right-hand side above we use the bound (H1) in Assumption (\ref{assump:H_V_tail}), as well as the bound
\[\forall k \in \mathbb{N}, \quad \max_{|x| \leq k} H(x) \leq C(\omega) \left(1 +  k^{\frac{1}{\alpha_1^{-}}}\right), \]
(see Lemma \ref{lem:estimate_max_heavy_rv} in the Appendix), which yield

\[\begin{split}
E^\omega_0[(X^\ast_1(t \wedge \tau))^2] &\leq C(\omega) \, E^\omega_0 \left[t + \int_0^{t} X^\ast_2(s \wedge \tau)^{\frac{1-\alpha_1^{-}}{\alpha_1^{-}}} \, ds \right] \\
&\leq C(\omega) \left( t +  \int_0^{t} E^\omega_0[X^\ast_2(s \wedge \tau)^{2 \varepsilon_1^+}] \right) \, ds\end{split}\]
with $\varepsilon_1^+$ as in \eqref{def:epsilon+}. Since by assumption $\varepsilon_1 < 1$, recalling the definition \eqref{def:epsilon}, we see that $\varepsilon_1^+ < 1$ provided $\alpha_1^{-}$ is sufficiently close to $\alpha_1$. Using Jensen's inequality, we therefore have 
$$E^\omega_0[X^\ast_2(s \wedge \tau)^{2 \varepsilon_1^+}] \leq E^\omega_0[X^\ast_2(s \wedge \tau)^{2}]^{\varepsilon_1^+} = (v_2(t))^{\varepsilon_1^+}.  $$
Thus, we get
\[h_2(t) = E^\omega_0[(X^\ast_1(t \wedge \tau))^2]  \leq C(\omega) \left( t +   \int_0^{t}  (v_2(s))^{\varepsilon_1^+} \, ds \right). \]
In a similar we obtain the symmetric bound
\[v_2(t) \leq C(\omega) \left( t +   \int_0^{t}  (h_2(s))^{\varepsilon_2^+} \, ds \right). \]
We thereby get the system of inequalities
\begin{equation*}
\begin{cases}
h_2(t) \leq C(\omega) \left(t +   \int_0^{t}  (v_2(s))^{\varepsilon_1^+} \, ds \right) \\
v_2(t) \leq C(\omega) \left(t +   \int_0^{t}  (h_2(s))^{\varepsilon_2^+} \, ds \right). 
\end{cases}
\end{equation*}
Since $\varepsilon_1^+, \varepsilon_2^+ < 1$, we get that $h_2(t), v_2(t) < \infty$ for all $t \geq 0$. More specifically, by the first point in Lemma \ref{lem:bound_sol_diff_ineq} below (see subsection \ref{app:proof_estimates_system} in the Appendix) with $\kappa=T$, we obtain 
\[\begin{split}
\forall t \in [0,T], \qquad  E^\omega_0[X_1^\ast(t\wedge \tau)^2] = h_2(t) &\leq C(\omega) (T^{1/A_1} + T)^{A_1} \\
&\leq C(\omega) \left(T + T^{A_1} \right),
\end{split}\]
(note that $A_1 >1$), and likewise
\[\forall t \in [0,T], \qquad E^\omega_0[X_2^\ast(t \wedge \tau)^2] = v_2(t) \leq C(\omega) \left(T + T^{A_2} \right). \]
Note that these bounds are uniform in the cut-off parameter $R$ defining the stopping time $\tau$, which readily ensures non-explosion in $[0,T]$. Sending $R \to \infty$, we also get by Fatou's lemma the bounds
\begin{equation}
\label{eq:non_explosion_estimates}
\forall t \in [0,T], \qquad  E^\omega_0[X_1^\ast(t)^2] \leq C(\omega) \left(T+ T^{A_1}\right) \quad \hbox{and} \quad  E^\omega_0[X_2^\ast(t)^2] \leq C(\omega) \left(T+ T^{A_2}\right),
\end{equation}
which yield the claim.



\textbf{ Second case:} Assume now that $\varepsilon_2 <1$ and $\varepsilon_1  \in \left(1, \frac{1}{\varepsilon_2} \right)$. Set $p := \frac{2}{\varepsilon_2^+} >2$. For all $R>0$, applying Dynkin's formula to the process $\varphi_R(X_t)$, where $\varphi_R(x) := \mathbf{1}_{\{\|x\| \leq R\}} \|x\|^p$, and noting that, for all $x \in \mathbb{Z}^2$ with $\|x\|<R$ we have  
\[\mathcal{L}^{\omega^\ast} \varphi_R(x) = H(x_2)^{\alpha_1^{-}} \, V(x_1)^{1-\alpha_2^{-}} 
\left(|x_2+1|^p + |x_2-1|^p - 2 |x_2(s)|^p\right), \]
then setting again $\tau = \tau_R$ we obtain that the process
\[ |X^\ast_2(t \wedge \tau)|^p  - \int_0^{t \wedge \tau} H(X^\ast_2(s))^{\alpha_1^{-}} \, V(X^\ast_1(s))^{1-\alpha_2^{-}} 
\left(|X^\ast_2(s)+1|^p + |X^\ast_2(s)-1|^p - 2 |X^\ast_2(s)|^p\right) \, ds \]
is a bounded martingale. 
Taking expectations as above, we obtain
\[\begin{split}  
&E^\omega_0[|X^\ast_2(t \wedge \tau)|^p] \leq \\
&E^\omega_0 \left[ \int_0^{t \wedge \tau} H(X^\ast_2(s))^{\alpha_1^{-}} \, V(X^\ast_1(s))^{1-\alpha_2^{-}} \left(|X^\ast_2(s)+1|^p + |X^\ast_2(s)-1|^p - 2 |X^\ast_2(s)|^p\right) \, ds \right].
\end{split}\]
Now, using the bound (H1) in Assumption (\ref{assump:H_V_tail}), as well as the bound
\[\forall k \in \mathbb{N}, \quad \max_{|x| \leq k} V(x) \leq C(\omega) \left(1 +  k^{\frac{1}{\alpha_2^{-}}}\right), \]
see equation (A.1) in \cite{Andres:2016aa}, we get 
\[\forall s \in [0,T], \quad H(X^\ast_2(s))^{\alpha_1^{-}} \, V(X^\ast_1(s))^{1-\alpha_2^{-}} \leq C(\omega) \, |X^\ast_1(s)|^{2 \varepsilon^{+}_2}.  \]
Further using the inequality
\[ |x+1|^p - 2 |x|^p + |x-1|^p  \leq c \left(|x|^{p-2} + 1 \right), \qquad x \in \mathbb{Z}, \]
with a constant $c$ depending only on $p$, we obtain
\[\begin{split}
E^\omega_0[|X^\ast_2(t)|^p] &\leq C(\omega) \int_0^t E^\omega_0 \left[ |X^\ast_1(s \wedge \tau)|^{2 \varepsilon^{+}_2} \left\{|X^\ast_2(s \wedge \tau)|^{p-2} + 1 \right\} \right] \, ds \\
&\leq C(\omega) \left(t +  \int_0^t E^\omega_0 \left[ |X^\ast_1(s \wedge \tau)|^{2 \varepsilon^{+}_2} |X^\ast_2(s \wedge \tau)|^{p-2} \right] \, ds \right). 
\end{split}\]
We now use H\"{o}lder's inequality with the pair of conjugate exponents $(q,q')$ given by 
\[\frac{1}{q} = \varepsilon^+_2, \qquad \frac{1}{q'} = \frac{p-2}{p} = 1- \varepsilon^+_2 \]
(note that $q,q' >1$ since $\varepsilon_2^+ \in (0,1)$). This allows to bound the integrand in the last integral above by
\[ E^\omega_0[|X^\ast_1(s \wedge \tau)|^{2}]^{\varepsilon^+_2} E^\omega_0[|X^\ast_2(s \wedge \tau)|^{p}]^{1-\varepsilon^+_2} = h_2(s)^{\varepsilon^+_2} v_p(s)^{1-\varepsilon^+_2}. \]
Thus we get
\begin{equation}
\label{eq:bd.vp.sndcase}
v_p(t) \leq C(\omega) \left(t +  \int_0^t h_2(s)^{\varepsilon^+_2} v_p(s)^{1-\varepsilon^+_2} \, ds \right). 
\end{equation} 
On the other hand, reasoning as in the \textbf{First Case} above, we get
\[h_2(t) \leq C(\omega) \left( t +  \int_0^{t} E^\omega_0[X^\ast_2(s \wedge \tau)^{2 \varepsilon_1^+}] \, ds \right),\]
where the last integrand is bounded, by Jensen's inequality, by 
$$E^\omega_0[X^\ast_2(s \wedge \tau)^{p}]^{\frac{2 \varepsilon_1^+}{p}} = v_p(s)^{\varepsilon_1^+ \varepsilon_2^+},$$
so that 
\begin{equation}
\label{eq:bd.h2.sndcase}
h_2(t) \leq C(\omega) \left( t +  \int_0^{t} v_p(s)^{\varepsilon_1^+ \varepsilon_2^+} \, ds \right). 
\end{equation} 
Combining \eqref{eq:bd.vp.sndcase} and \eqref{eq:bd.h2.sndcase} we deduce, by the second point of Lemma \ref{lem:bound_sol_diff_ineq} applied with $\kappa=T$, that for all $t \in [0,T]$,
\[E^\omega_0[X_1^\ast(t\wedge \tau)^2] = h_2(t) \leq C(\omega)\left(T+T^{A_1}\right)  \quad \hbox{and} \quad E^\omega_0[X_2^\ast(t \wedge \tau)^2] = v_2(t) \leq C(\omega) \left(T+ T^{B_2}\right), \]
with $B_2 = \frac{A_2}{\varepsilon_2^+}$. Again, these bounds are uniform in the cut-off parameter $R$ defining the stopping time $\tau$, which ensures non-explosion in $[0,T]$, and
sending $R \to \infty$, we also get by Fatou's lemma
\[\forall t \in [0,T], \qquad  E^\omega_0[X_1^\ast(t)^2] \leq C(\omega) \left(T+ T^{A_1}\right) \quad \hbox{and} \quad  E^\omega_0[X_p^\ast(t)^2] \leq C(\omega) \left(T+T^{B_2}\right), \]
which yield the claim.

\textbf{Last case}: If $\varepsilon_1 <1$ and $\varepsilon_2  \in \left(1, \frac{1}{\varepsilon_1} \right)$, then reasoning similarly as in the previous case, we get non explosion and 
\[\forall t >0, \qquad E^\omega_0[ (X^\ast_1(t))^2 + (X^\ast_2(t))^p]  < \infty \]
where $p := \frac{2}{\varepsilon_1^+}$. \end{proof}

As a by-product of the non-explosion results obtained above, one is able to control the long-time behaviour of the random walk. We assume that 
the non-explosion criterion $\varepsilon_1 \varepsilon_2 < 1$ is satisfied. 

\begin{proposition}
\label{prop:over_scaling}
Define the rescaled random walk $(X^{(T)}(t))_{t \geq 0}$ as follows:
$$X^{(T)}_i(t) = T^{-a_i/2} X_i(Tt), \quad i=1,2$$ 
with $a_i, i =1,2$, two positive constants such that $a_i > \frac{1+ \varepsilon_i}{1- \varepsilon_1 \varepsilon_2}$, $i=1,2$. Then, for $\mathbb{P}$ a.e. $\omega \in \Omega$, for all $t>0$, and $i=1,2$, 
\[ \sup_{0 \leq s \leq t} X_i^{(T)}(s) \underset{T \to \infty}{\longrightarrow} 0 \]
in $P^{\omega}_0$-probability.
\end{proposition}

\begin{remark}
This Proposition shows that $ \frac{1+ \varepsilon_i}{2(1- \varepsilon_1 \varepsilon_2)}$ are upper bounds for the correct scaling exponent for $X_i$, $i=1,2$. In the special case where $\alpha_1>1$ or $\alpha_2 >1$, then we actually have a scaling limit as shown in the next section. For instance 
for $\alpha_1 < 1$ and $\alpha_2 > 1$, the correct scaling exponents are $\frac{\delta}{2} = \frac{1}{4}+ \frac{1}{4 \alpha_1}$ and $1/2$, respectively. Note
that in this regime $\varepsilon_1 = \frac{1}{2}(1- \frac{1}{\alpha_1})$ and $\varepsilon_2= 0$, so these scaling exponents are consistent with the above proposition. 
In the other regimes, we also expect the correct scaling exponents to be 
$$ \frac{1+ \varepsilon_i}{2(1- \varepsilon_1 \varepsilon_2)}, \qquad i =1,2,  $$
but our methods fall short of proving tightness with these exponents.
\end{remark}

\begin{proof}
First note that, the proof of Lemma \ref{lem:Xstar_is_conservative} has shown that $(X_i^\ast(t))_{t \geq 0}$, for $i =1,2$,  are $L^2$ martingales with respect to the canonical filtration $\mathcal{F}_t := \sigma (\{X^\ast(s):  s \leq t \})$. Let us now assume for instance that $\varepsilon_1<1$ and $\varepsilon_2 < 1$ (the other cases are similar). The processes $(X_i^\ast(t))_{t \geq 0}$ being c\`{a}dl\`{a}g martingales, by Doob's inequality, and using the bounds \eqref{eq:non_explosion_estimates}, we get, for all $t>0$, 
$$E^\omega_0 \left[ \sup_{0 \leq s \leq t} X_i^\ast(s)^2 \right] \leq 2 E^\omega_0 \left[ X_i^\ast(t)^2 \right] \leq C(\omega) (t+ t^{A_i}), \qquad i \in \{1,2\}, $$
for $A_i$, $i =1,2$, of the form $A_i = \frac{1+ \varepsilon_i^+}{1- \varepsilon_1^+ \varepsilon_2^+}$ where $\varepsilon_1^+ > \varepsilon_1$ and $\varepsilon_2^+ > \varepsilon_2$ (see \eqref{def:ai}). Moreover, choosing $\varepsilon_i^+$ sufficiently close to $\varepsilon_i$, we can further ensure that $a_i > A_i$.  After rescaling, we obtain
$$E^\omega_0 \left[\left( \sup_{0 \leq s \leq t} T^{-a_i/2} |X_i^{\ast}(T s)| \right)^2\right] \leq C(\omega) (T^{1-a_i} t + t^{A_i} T^{A_i-a_i}) \qquad i \in \{1,2\}.$$
Since $a_i >A_i>1$, setting
$$X^{\ast (T)}_i(t) = T^{-a_i/2} X_i^\ast(Tt), \quad i=1,2,$$
then as $T \to \infty$, we obtain that $X^{\ast (T)}_i$ converges to $0$ in probability, in $D([0,\infty))$.  Now, we can write
\[X^{(T)}_i = X^{\ast (T)}_i \circ A^{(T)},\]
where
\[A^{(T)}(t) := \frac{1}{T} A(Tt) = \frac{1}{T} \int_0^{Tt} H(X_2(s))^{\alpha_1^{-}} V(X_1(s))^{\alpha_2^{-}} \, ds, \qquad t \geq 0. \]
By \eqref{eq:time_change_admiss} and the ergodicity of the environment seen from the particle, for $\mathbb{P}$ a.e. $\omega$, for all $t \geq 0$, 
\begin{equation}
\label{eq:cv_time_change_x_star}
A^{(T)}(t) \underset{T \to \infty}{\longrightarrow} c t,
\end{equation}
$P^\omega_0$ a.s., where $c :=  \mathbb{E}[H(0)^{\alpha_1^{-}}] \, \mathbb{E}[V(0)^{\alpha_2^{-}}] < \infty$. Indeed, recalling the notation of \eqref{eq;env_seen_from_partic}, we can write 
\[A^{(T)}(t) = \frac{1}{T} \int_0^{Tt} [\omega(s)](0,e_1)^{\alpha_1^{-}} \, [\omega(s)](0,e_2) ^{\alpha_2^{-}}  \, ds, \]
and, since 
\[\mathbb{E}[\omega(0,e_1)^{\alpha_1^{-}} \, \omega(0,e_2)^{\alpha_2^{-}}] =  \mathbb{E}[H(0)^{\alpha_1^{-}}] \, \mathbb{E}[V(0)^{\alpha_2^{-}}] =c,  \] 
the claimed convergence holds by Birkhoff's ergodic theorem applied to the process $(\omega(t))_{t \geq 0}$. Since the processes $A^{(T)}$ are non-decreasing, by Dini's theorem we deduce that $A^{(T)} \underset{T \to \infty}{\longrightarrow} I$ a.s. in $C([0,\infty))$, where $I(t) := c t$, $ t \geq 0$. By Lemma \ref{lem:cont_comp}, we deduce that $X^{(T)}_i = X^{\ast (T)}_i \circ A^{(T)}$ converges in probability to $0$ as $T \to \infty$, for $i=1,2$.
\end{proof}

\section{Scaling limit results when $\max(\alpha_1,\alpha_2)>1$}
\label{sec:scaling}

In this section we assume that $\max(\alpha_1,\alpha_2)>1$. We will prove scaling limit results for the properly rescaled random walk. We start with the easy case where $\alpha_1, \alpha_2 >1$ and then tackle the interesting case where $\min(\alpha_1, \alpha_2) < 1$.

\subsection{Case $\alpha_1 > 1$, $\alpha_2 >1$.}

In this case we have $\mathbb{E}[H(x)] := c_H < \infty$ and $\mathbb{E}[V(x)] := c_V < \infty$. Invariance principle hence follows by standard results, see \cite{BISKUP}. In our specific framework, it can be derived along the following lines. Reasoning as for the proof of the convergence \eqref{eq:cv_time_change_x_star} above, by ergodicity of the environment seen from the walker, for all $t \geq 0$ we get, $\mathbf{P}$ almost-surely, 
\[\frac{1}{T} \int_0^{Tt} H(X_2(s)) \, ds  \underset{T \to \infty}{\longrightarrow} c_H \, t, \qquad \frac{1}{T} \int_0^{Tt} V(X_1(s)) \underset{T \to \infty}{\longrightarrow} c_V \, t. \]  
Now, setting $X^{(T)}(t) := T^{-1/2} (X_1(Tt),X_2(Tt))$, $t \geq 0$, we therefore have, for all $t>0$, the $\mathbf{P}$ almost-sure convergence 
\[\langle X^{(T)}_1, X^{(T)}_1 \rangle (t) \underset{T \to \infty}{\longrightarrow} 2 c_H t,  \qquad  \langle X^{(T)}_2, X^{(T)}_2 \rangle (t) \underset{T \to \infty}{\longrightarrow} 2 c_V t.\]
Since moreover $\langle X^{(T)}_1, X^{(T)}_2 \rangle (t) = 0$ for all $T>0$, we deduce by standard arguments that, for $\mathbb{P}$ a.e. $\omega$, $(X^{(T)}(t))_{t \geq 0}$ converges under $P^\omega_0$, in  distribution, for the Skorohod topology, to a $2$-dimensional Brownian motion $(\beta(t))_{t \geq 0}$ with diffusion matrix
$$\Sigma = \begin{pmatrix}
2 c_H & 0 \\
0  & 2 c_V
\end{pmatrix}. $$

\subsection{Case where either $\alpha_1<1$ or $\alpha_2 < 1$.}

Now we work in the semi-degenerate case where either $\alpha_1<1$ or $\alpha_2 < 1$. Let us assume for instance $\alpha_1<1$ and $\alpha_2 >1$.

\subsubsection{The result}

In this case we also have $\mathbb{E}[V(x)] := c_V < \infty$, whence, reasoning as before, 
$$\langle X^{(T)}_2, X^{(T)}_2 \rangle (t) \underset{T \to \infty}{\longrightarrow} 2c_V t, \qquad a.s. $$
Hence, $(X_2^{(T)}(t))_{t \geq 0}$ converges $\mathbb{P}$-a.s., in distribution, to a $1$-dimensional Brownian motion $(\beta_1(t))_{t \geq 0}$ with diffusion coefficient $\sqrt{2}$. However now $\mathbb{E}[H(x)]=\infty$, thus a similar convergence fails to hold for $(X_1^{(T)}(t))_{t \geq 0}$. We actually have a superdiffusive behavior of the fast component $X_1$, with a non-Gaussian scaling limit. In order to obtain a quenched formulation, we will make the following assumption. 

\begin{assumption}
\label{assump:H_V_tail_bis}
  We assume:
  \begin{enumerate}
  \item The variables $V(x)$, $x \in \mathbb{Z}$ are i.i.d, with finite mean, and there exists $c>0$ such that $V(x) \geq c$ for all $x \in \mathbb{Z}$.
  \item There is an underlying double-sided stable subordinator $(\mathcal{H}(x))_{x \in \mathbb{R}}$ of index $\alpha_1$, independent of $(V(x))_{x \in \mathbb{Z}}$, such that the random variables $H(x)$ are given by
 \begin{equation}
 \label{eq:relation_env_stable_proc}
 \forall x \in \mathbb{Z}, \quad H(x) = \mathcal{H}(x+1) - \mathcal{H}(x).
 \end{equation}
  \end{enumerate}
\end{assumption}

For all $T>0$, we define 
\begin{equation}
\label{eq:def_env_rescaled}
H^{(T)}(x) := T^{\frac{1}{2\alpha_1}} \left( \mathcal{H} \left( \frac{x+1}{\sqrt{T}} \right) - \mathcal{H} \left(\frac{x}{\sqrt{T}} \right) \right), \qquad x \in \mathbb{Z}.  
\end{equation}
Note that $(H^{(T)}(x))_{x \in \mathbb{Z}}$ has the same distribution as the original collection $(H(x))_{x \in \mathbb{Z}}$, thanks to the scaling property of $\mathcal{H}$. In the sequel it will be convenient to use the $H^{(T)}(x)$ rather than the $H(x)$ in the definition of the rescaled process $X^{(T)}$ so as to get a quenched result. Namely, for all $T>0$, we  define $X^{T}$ the random walk started from $0$ in the environment defined by the variables $H^{(T)}(x)$ and $V(x)$, and then set $X^{(T)}(t) := (X_1^{(T)}(t), X_2^{(T)}(t))$, where 
\begin{equation}
\label{eq:def_rescaled_process_rescaled_environment}
\forall t \geq 0, \qquad X_1^{(T)}(t) := \frac{1}{T^\delta} X^{T}_1(Tt), \quad  X_2^{(T)}(t) := \frac{1}{T} X^{T}_2(Tt).
\end{equation} 
Here $\delta = \frac{1}{2}\left(1 + \frac{1}{\alpha_1}\right)$. 

%

\begin{remark}
Redefining the variables $H(x)$ as \eqref{eq:def_env_rescaled} for each value of $T$ allows to obtain quenched result. However this does not change the law of the environment, so without this modification all the scaling limit results stated below would remain true, but in annealed - rather than quenched - sense.
\end{remark}

\begin{theorem}
\label{thm:conv}
For $\mathbb{P}$-a.e. realisation of the environment, there holds the convergence in law, in the space $C([0,\infty))$,
\begin{equation}
\label{eq:conv.clock_proc_delta} 
\left( \frac{1}{T^\delta} \int_0^{Tt} H^{(T)} (X^T_2(s)) \, ds \right)_{ t \geq 0} \underset{T \to \infty}{\Longrightarrow} \left(\Delta(t)\right)_{ t \geq 0},
\end{equation}
where $(\Delta(t))_{t \geq 0}$ is the Kesten-Spitzer process, defined as 
\begin{equation}
\label{eq:def_ks_process_thm}
\Delta(t) = \Delta^\mathcal{H}(B_2)(t) = \int_{-\infty}^{+\infty} L^{B_2}(t,x) \, d\mathcal{H}(x),  \quad t \geq 0,
\end{equation}
where $(L^{B_2}(t,x))_{x \in \mathbb{R}, \, t \geq 0}$ is the local time process of $B_2$, where $B_2$ is a Brownian motion in $\mathbb{R}$ started from $0$ and with diffusion coefficient $\sqrt{2}$.
In addition, for $\mathbb{P}$-a.e. realisation of the environment, we have the convergence in law, in $D([0,\infty))$, 
\begin{equation}
\label{eq:conv_x_statement} 
\left( X^{(T)}(t) \right)_{t \geq 0} \underset{T \to \infty}{\Longrightarrow} \left(B_1(\Delta(t)), B_2(t)) \right)_{t \geq 0},  
\end{equation}
where $(\Delta(t))_{t \geq 0}$ is as in \eqref{eq:def_ks_process_thm}, and $B_1$ is a Brownian motion started from $0$, with diffusion coefficient $\sqrt{2}$, and independent from $B_2$.
\end{theorem}

\begin{remark}
We believe the second point of Assumption \ref{assump:H_V_tail_bis} may be relaxed to just assuming $\mathbb{P}(H(x_2)>L) \underset{L \to \infty}{=} c_H L^{-\alpha_1(1 + o(1))}$, at the expanse of re-defining the variables $H^{(T)}(x)$ accordingly: in this more general case, as in \cite{benarous2005} (see Section 4 therein), we can define $G:[0,\infty) \to [0,\infty)$ such that, for all $a \geq 0$
\[ \mathbb{P}(\mathcal{H}(1) > G(a)) = \mathbb{P}(H(0)>a). \]
Such a function $G$ is well-defined since $Z(1)$ has a continuous distribution. Additionally, it is nondecreasing and right continuous, and we denote by $G^{-1}$ its right-continuous generalized inverse. For all $T>0$ we could then define 
\[H^{(T)}(x) := G^{-1} \left( T^{\frac{1}{2\alpha_1}} \left( \mathcal{H} \left( \frac{x+1}{\sqrt{T}} \right)- \mathcal{H} \left(\frac{x}{\sqrt{T}} \right) \right) \right) , \qquad x \in \mathbb{Z}.  \]
With this properly rescaled environment, the above scaling limit theorem is believed to remain true. However, in order to simplify the proofs, we shall stick here to the, already non trivial, case of $H(x)$ given as in \eqref{eq:relation_env_stable_proc}.  
\end{remark}

In the remainder of this section we prove the above theorem, first in the easier case where the random variables $V(x)$ are deterministic, second in the general case.

\subsubsection{Case where $V$ is constant}

Let us assume in a first step that $V$ is a deterministic constant, e.g. $V(x)=1$ for all $x \in \mathbb{Z}$. Thus the environment is given by

\begin{align*}
  \omega(x,x +e_1) &= \omega(x,x -e_1) = H(x_2)(\omega) > 0, \\
  \omega(x,x +e_2) &= \omega(x,x -e_2) = 1,
\end{align*}
In this case, $X(t)=(X_1(t),X_2(t))$ can be represented as $(S_1(A(t)), S_2(t))$, where $S(t) = (S_1(t),S_2(t))$ is a 2-dimensional random walk jumping to each neighboring site at rate 1, and 
\[A(t) = \int_0^t H(S^2(r)) \, dr, \quad t \geq 0.\]
Note that $S^2(t)$, $t \geq 0$, is independent of the random variables $H(x)$, $x \in \mathbb{Z}$, so $(A(t))_{t \geq 0}$ is a continuous-time random walk in random scenery whose scaling limit is given by a 
Kesten-Spitzer process like \eqref{eq:def_ks_process_thm}, as proven in \cite{kesten1979limit}. Hence the following lemma is a consequence of the results of \cite{kesten1979limit} (with the slight difference that in that article
the random walk is in discrete time), however we shall provide a different proof which more easily generalises to the case where the random variables $V(x)$, $x \in \mathbb{Z}$, are no longer constant.

\begin{lemma}
\label{lem:conv_constant_case}
For $\mathbb{P}$ a.e. realisation of the environment, we have the convergence in law in $C([0,\infty))$
\[\left(\frac{1}{T^\delta} \int_0^{Tt} H^{(T)}(S_2(r)) \, dr \right)_{t \geq 0} \underset{T \to \infty}{\Longrightarrow} (\Delta(t))_{t \geq 0}, \]
where  $(\Delta(t))_{t \geq 0}$ is as in \eqref{eq:def_ks_process_thm}.
\end{lemma}

\begin{proof}
Since $(S(t))_{t \geq 0}$ is a simple random walk with quadratic variations 
\[ \langle S_1,S_1 \rangle_t = \langle S_2,S_2 \rangle_t = 2 t,\]
by Donsker's theorem we have the convergence in law, in $D([0,\infty))$,
\begin{equation}
\label{eq:donsker_s} 
S^{(T)}(t) := \frac{1}{\sqrt{T}} S(Tt) \underset{T \to \infty}{\Longrightarrow} B(t) 
\end{equation}
where $(B(t))_{t \geq 0}$ is a two-dimensional Brownian motion with diffusion coefficient $\sqrt{2}$. In addition, we have convergence at the level of the local times thanks to Lemma \ref{lem:conv_lt_srw} (see subsection \ref{subsec:appendix_local_times} of the Appendix).
Namely, setting 
\[\ell^T_t(x) := \frac{1}{\sqrt{T}} \int_0^{Tt} \mathbf{1}_{\{S^2(r) = \lfloor T x \rfloor\}} \, dr, \qquad t \geq 0, \quad x \in \mathbb{R}, \]
as well as $\tilde{\ell}^T$ the version of these local times obtained by linear interpolation in space
 \[\tilde{\ell}^{\xi, T}_t(x) = \left(\lceil \sqrt{T} x \rceil - \sqrt{T} x\right)\ell^{T}_t(x) + \left(\sqrt{T} x - \lfloor \sqrt{T} x \rfloor \right) \ell^{T}_t(x+1),  \qquad t \geq 0,  \quad x \in \mathbb{R},\]
 then, jointly with the convergence \eqref{eq:donsker_s}, we have the convergence in law in $C([0,\infty) \times \mathbb{R}, \mathbb{R})$, for the topology of uniform convergence on compact sets, 
\[ (\ell^T_t(x))_{x \in \mathbb{R}, t \geq 0} \overset{(d)}{\underset{T \to \infty}{\Longrightarrow}} (L^{B_2}(t,x))_{x \in \mathbb{R}, t \geq 0} \]
where $L^{B_2}(t,x)$ are the local times of $B_2$, see Lemma \ref{lem:conv_lt_srw} below. On the other hand, setting 
\[\mu^T := \frac{1}{T^{1/2\alpha_1}} \sum_{x \in T^{-1/2} \mathbb{Z}} H^{(T)}(\sqrt{T}x) \, \delta_{x}, \qquad T>0, \]
then, recalling \eqref{eq:def_env_rescaled}, for all $T>0$ we can re-express $\mu_T$ as 
$$\mu_T = \sum_{x \in T^{-1/2} \mathbb{Z}} \left(\mathcal{H} \left(x+ \frac{1}{\sqrt{T}} \right) - \mathcal{H}(x)\right) \, \delta_{x}.$$ 
Thus, almost-surely, the sequence of measures $(\mu_T)_{T >0}$ converges vaguely, as $T\to \infty$, 
 to the Stieltjes measure $\mathcal{H}(dx)$ of the stable subordinator $\mathcal{H}$ defining the environment. Given a large number $M>0$, let us now consider the map
\[\Phi : \begin{cases}
        C_{M}(\mathbb{R}_+\times \mathbb{R}) \times \mathcal{M}(\mathbb{R}) &\to C(\mathbb{R}_+) \\
        (f,\mu) &\mapsto (t \mapsto \int_{\mathbb{R}} f(t,x) \, \mu(dx))
    \end{cases} \]
where 
$$C_M(\mathbb{R}_+\times \mathbb{R}) := \{f \in C(\mathbb{R}_+\times \mathbb{R}): \forall t \geq 0, \hbox{supp}(f(t,\cdot)) \subset \subset [-M-1,M+1] \}$$
is the space of continuous functions on $\mathbb{R}_+ \times \mathbb{R}$ that are compactly supported in space in the interval $[-M-1,M+1]$. We equip $C_M(\mathbb{R}_+\times \mathbb{R})$ and $C(\mathbb{R}_+)$ with the topology of uniform convergence on compact subsets (of $\mathbb{R}_+ \times \mathbb{R}$ and $\mathbb{R}_+$, respectively). Then the map $\Phi$ is continuous. 
Therefore, taking a cut-off function $\varphi_M \in C(\mathbb{R})$ such that $\varphi_M(x)=0$ if $|x|>M+1$ and $\varphi_M(x) = 1$ if $|x| \leq M$, we have the quenched convergence in distribution
\[\begin{split}
\int_\mathbb{R} \varphi_M(x) \tilde{\ell}^T_t(x) \, \mu_T(dx) \underset{T \to \infty}{\Longrightarrow} 
\int_{\mathbb{R}} \varphi_M(x) L^{B_2}_t(x) \, \mathcal{H} (dx). 
\end{split}\]
Further note that $\tilde{\ell}^T_t(x) =  \ell^T_t(x)$ for all $x \in T^{-1/2} \mathbb{Z}$, hence for all $x$ in the support of $\mu_T$, so we also have
\[\begin{split}
\int_\mathbb{R} \varphi_M(x) \ell^T_t(x) \, \mu_T(dx) \underset{T \to \infty}{\Longrightarrow} 
\int_{\mathbb{R}} \varphi_M(x) L^{B_2}_t(x) \, \mathcal{H} (dx). 
\end{split}\]
Now
\begin{equation}
\label{eq:dec_integ_clock}
\begin{split}
&\frac{1}{T^\delta} \int_0^{Tt} H^{(T)}(S_2(r)) \, dr = \int_\mathbb{R} \ell^T_t(x) \, \mu_T(dx) \\ 
= &\int_\mathbb{R} \varphi_M(x) \ell^T_t(x) \, \mu_T(dx)  + \int_\mathbb{R} (1-\varphi_M(x)) \ell^T_t(x) \, \mu_T(dx),
\end{split}
\end{equation}
while 
\begin{equation}
\label{eq:dec_integ_limit}
\begin{split}
&\int_{\mathbb{R}} L^{B_2}(t,x) \, \mathcal{H}(dx) \\
&= \int_{\mathbb{R}} \varphi_M(x) L^{B_2}(t,x) \, \mathcal{H}(dx)  + \int_{\mathbb{R}} (1- \varphi_M(x)) L^{B_2}(t,x) \, \mathcal{H}(dx). 
\end{split}
\end{equation}
We know that the first term in the right-hand side of \eqref{eq:dec_integ_clock} converges in distribution, as $T \to \infty$, to the first term in the right-hand side of \eqref{eq:dec_integ_limit}, for the topology of uniform convergence on compact sets on $C(\mathbb{R})$. In addition, defining $\varepsilon(M)$ in a similar way as $\varepsilon(A)$ is defined in equation (2.17) of \cite{kesten1979limit}:
\[\varepsilon(M) := \sup_{T \geq 1} \, \mathbb{P} \left( \underset{0 \leq r \leq t}{\sup} |S^{(T)} (r)| > M \right), \]
and reasoning as in the aformentioned article, we see that $\varepsilon(M) \underset{M \to \infty}{\longrightarrow} 0$.  We deduce that for any fixed $t \geq 0$, and any $\epsilon>0$ we have, for $M$ large enough,
\[\forall T \geq 1,  \quad \mathbb{P} \left( \underset{0 \leq r \leq t}{\sup} |S^{(T)} (r)| > M \right) \leq \epsilon/2 \qquad \hbox{and} \qquad \mathbb{P} \left( \underset{0 \leq r \leq t}{\sup} |B(r)| > M \right) \leq \epsilon/2.\]
Since $1-\varphi_M$ vanishes on $[-M,M]$, the second terms in the right-hand sides of \eqref{eq:dec_integ_clock} and \eqref{eq:dec_integ_limit} therefore vanish simultaneously with probability at least $1-\epsilon$. Since $\epsilon$ was arbitrary, this shows that these terms tend to $0$ in probability as $M \to \infty$, uniformly in $T$. As a consequence, we finally obtain the convergence in law
$(A^{(T)}(t))_{t \geq 0} \underset{T \to 0}{\Longrightarrow} (\Delta(t))_{t \geq 0} $ for the topology of uniform convergence of compact sets on $C(\mathbb{R}_+)$, where
\[A^{(T)}(t) :=  \frac{1}{T^\delta} A(Tt), \quad \Delta(t) = \int_{\mathbb{R}} L^{B_2}(t,x) \, \mathcal{H}(dx), \qquad t \geq 0.\]
\end{proof}

From the above lemma we deduce the requested convergence. Indeed, setting as in \eqref{eq:def_rescaled_process_rescaled_environment},
$$X^{(T)}_1(t) :=  T^{-\delta/2} X^T_1(Tt), \quad X^{(T)}_2(t) : = \frac{1}{\sqrt{T}} X_2(Tt), \qquad t \geq 0,$$
(note that in our case $X_2^T = X_2= S_2$ does not depend on the rescaled environment), we have
\[X^{(T)}_1(t) =  S_1^{(T^\delta)} \circ A^{(T)} (t), \quad X^{(T)}_2(t) =  S_2^{(T)} (t),\]
where, for all $M>0$, $S^{(M)}_i(t) := \frac{1}{\sqrt{M}} S_i(Mt)$, $i=1,2$. Now, for $\mathbb{P}$-a.e. realisation of the environment we have the joint convergence in law 
$$(A^{(T)}, S_2^{(T)}) \underset{T \to \infty}{\Longrightarrow} (\Delta,B_2)$$ 
in $C\times D$ , as well as the convergence in law $S_1^{(T^\delta)} \underset{T \to \infty}{\Longrightarrow} B_1$ in $D$. Since, moreover, $S_1^{(T^\delta)}$ is independent of $(A^{(T)}, S_2^{(T)})$, these two convergences happen jointly. Using Lemma \ref{lem:composition} we deduce the convergence in law 
\[(X^{(T)}_1,X^{(T)}_2)  \underset{T \to \infty}{\Longrightarrow} (B_1\circ \Delta, B_2), \]
in $D\times D$, as requested.

\subsubsection{General case}

In the case where $V(x)$, $x \in \mathbb{Z}$, are no longer deterministic, the above proof no longer works. One difficulty is to obtain the convergence in law of the local times of $(X_2(t))_{t \geq 0}$ (when diffusively rescaled), to those of a Brownian motion: this is non-trivial as $(X_2(t))_{t \geq 0}$ is no longer a simple random walk. Another difficulty appearing is the interdependence of the two components of the walk, $X_1$ and $X_2$, which no longer allows to represent them using a two-dimensional simple random walk. To circumvent these obstacles, we will rely on appropriate time changes. 

\subsubsection{A time change}

We introduce a new generator. For all $f: \mathbb{Z}^2 \to \mathbb{R}$ and $y = (y_1,y_2) \in \mathbb{Z}^2$, we define

\begin{equation}
\label{eq:generator_y}
\begin{split}
\mathcal{L}^Y f(y) &= \frac{H(y_2)}{V(y_1)} (f(y+e_1)-2 f(y) + f(y-e_1)) \\
&+ (f(y+e_2)- 2f(y) + f(y-e_2)), 
\end{split}
\end{equation}
and let $Y(t) = (Y_1(t), Y_2(t))$, $t \geq 0$, denote the random walk generated by $\mathcal{L}^Y$. Note that (see e.g. exercice 2.32 in \cite{Liggett2010ContinuousTM}), $(Y(t))_{t \geq 0}$ may be constructed from our original random walk $(X(t))_{t \geq 0}$ via the time change 
\begin{equation}
\label{eq:timechange_x_to_y}
Y(t) = X(A^{-1}(t)), \quad A(t) := \int_0^t V(X_1(s)), \qquad t \geq 0.
\end{equation}
We also denote by $Y^T$ the process with generator as in \eqref{eq:generator_y} but with the random variables $H$ replaced by $H^{(T)}$ as defined in \eqref{eq:def_env_rescaled}. 
The advantage of performing this time change is that, for all $T>0$, the vertical component $(Y^T_2(t))_{t \geq 0}$ is now a simple random walk independent of the environment. In particular, convergence holds for its diffusively rescaled local times thanks to Lemma \ref{lem:conv_lt_srw} (see Appendix). 

Now, for all $T>0$, the first component $Y_1^T(t)$ has quadratic variation
\[\langle Y_1^T, Y_1^T\rangle_t = 2 \int_0^t \frac{H^{(T)}(Y_2^T(s))}{V(Y_1^T(s))} ds, \qquad t \geq 0.\]
Let us introduce processes $D^{V,T}$ and $D^T$ defined by
\begin{equation}
\label{eq:def_dvt}
D^{V,T}(t) := \frac{1}{T^\delta} \int_0^{Tt} \frac{H^{(T)}(Y_2^T(s))}{V(Y_1^T(s))} ds, \quad D^{T}(t) := \frac{1}{T^\delta} \int_0^{Tt} H^{(T)}(Y_2^T(s)) ds, \qquad t \geq 0.
\end{equation}
We have 
$$D^T(t) = \int_\mathbb{R} \tilde{\ell}^{Y,T}_t(x) \, d \mu^T(x)$$
where 
$$\mu^T(dx) = \frac{1}{T^{\frac{1}{2\alpha_1}}} \sum_{x \in T^{-1/2} \mathbb{Z}} H(\sqrt{T} x) \delta_{x} =  \sum_{x \in T^{-1/2} \mathbb{Z}} \left(\mathcal{H} \left(x+ \frac{1}{\sqrt{T}} \right) - \mathcal{H}(x)\right) \, \delta_{x}.$$
Note that $\tilde{\ell}^{Y,T}$ being the local times process of a (rescaled) simple random walk, its law under the annealed measure is independent of $\mathcal{H}$, whence, reasoning as in the proof of Lemma \ref{lem:conv_constant_case}, we obtain the quenched convergence in law, in $C$, $(D^T(t))_{t \geq 0} \underset{T \to \infty}{\Longrightarrow} (\Delta(t))_{t \geq 0}$, where $(\Delta(t))_{t \geq 0}$ is the Kesten-Spitzer process defined in \eqref{eq:def_ks_process_thm}. We therefrom deduce that $(D^{V,T}(t))_{t \geq 0} \underset{T \to \infty}{\Longrightarrow} (\Delta(t))_{t \geq 0}$ also, thanks to the following lemma:

\begin{lemma}
\label{lem:cv_ratio_dv_d}
    For all fixed $t > 0$, we have the almost-sure convergence
    \begin{equation}
    \label{eq:ratio_dv_d}
    \frac{D^{V,T}(t)}{D^T(t)} = \frac{\int_0^{Tt} \frac{H^{(T)}(Y_2^T(s))}{V(Y_1^T(s))} \, ds}{\int_0^{Tt} H^{(T)}(Y_2^T(s)) \, ds}\underset{T \to \infty}{\longrightarrow} 1.
    \end{equation}    
\end{lemma}

\begin{proof}
    We have 
    \[\frac{D^{V,T}(t)}{D^T(t)} = \frac{\int_0^{Tt} \frac{H^{(T)}(Y_2^T(s))}{V(Y_1^T(s))} \, ds}{\int_0^{Tt} H^{(T)}(Y_2^T(s)) \, ds} = \frac{E(Tt)}{\int_0^{E(Tt)} V(Y_1^T(\tau_u)) \, du}, \]
    where $E(u):= \int_0^u \frac{H^{(T)}(Y_2^T(r))}{V(Y_1^T(r))} \, dr$ and we performed the change of variable $s=\tau_u$, with 
    \[\tau_u := \inf \{v \geq 0: E(v) > u\}. \]
    Now, the process $(Y_1^T(\tau_u),Y_2^T(\tau_u))$, $u \geq 0$, has generator given, for all $f: \mathbb{Z}^2 \to \mathbb{R}$ and $y =
    (y_1,y_2) \in \mathbb{Z}^2$, by
    \[\begin{split}
    \tilde{\mathcal{L}}^Y f(y) &:= (f(y + e_1) - 2f(y) + f(y-e_1)) \\ 
    &+ \frac{V(y_1)}{H^{(T)}(y_2)} (f(y + e_2) - 2f(y) + f(y-e_2)).
    \end{split}\]
    Note that, for all fixed $\omega \in \Omega$, under $P^\omega_0$, $(Y^T_1(\tau_u))_{u \geq 0}$ has the law of a simple random walk. In particular, the process $(Y^T_1(\tau_u))_{u \geq 0}$ is independent of the environment. We thus obtain 
    \[ \frac{D^{V,T}(t)}{D^T(t)} = \frac{E(Tt)}{\int_0^{E(Tt)} V(S(u)) \, du},\]
    where $(S(u))_{u \geq 0}$ is a simple random walk independent of $(V(x))_{x \in \mathbb{Z}}$. Since the $V(x)$ are i.i.d. random variables, 
      by an application of the ergodic theorem, we have $\mathbb{P}(d \omega) \times P^\omega_0$-a.s.,
    \[ \frac{1}{M} \int_0^M V(S(u)) \, du \underset{M \to \infty}{\longrightarrow} \mathbb{E}[V(0)] =1.\]
    Therefore, a.s.,  
    \[\frac{D^{V,T}(t)}{D^T(t)} = \frac{E(Tt)}{\int_0^{E(Tt)} V(S(u)) \, du} \underset{M \to \infty}{\longrightarrow} 1\]
    as claimed.
\end{proof}

As a consequence,
\begin{corollary}
 $(D^{V,T}(t))_{t \geq 0}$ converges in law in $C$, as $T \to \infty$, to the Kesten-Spitzer process $(\Delta(t))_{t \geq 0}$. 
 \end{corollary}
\begin{proof}
 This follows from the convergence in law $(D^{T}(t))_{t \geq 0} \underset{T \to \infty}{\Longrightarrow} (\Delta(t))_{t \geq 0}$, together with the following two facts:
\begin{enumerate}
\item[(i)] The family of processes $D^{V,T}$, $T\geq 1$, is tight in $C$ since, for all $t,s \geq 0$ with $s<t$
$$|D^{V,T}_t - D^{V,T}_s| \leq \frac{1}{c} \frac{1}{T^\delta} \int_0^{Tt} H^{(T)}(Y_2^T(r)) dr = \frac{1}{c} (D^T(t) - D^T(s)), $$
where $c$ is as in \eqref{assump:H_V_tail}, so, in view of the tightness criterion of \cite[Theorem 7.3]{billingsleySndEdition}, the claimed tightness is implied by the tightness of $D^T$, $T\geq 1$, 
\item[(ii)] By Lemma \ref{lem:cv_ratio_dv_d} and Slutsky's Lemma, the finite-dimensional marginals of $D^{V,T}(t)$ have the same limit in law as those of $D^T$.
\end{enumerate}
\end{proof}

Thus, recalling that $Y^{(T)}_1(t) := \frac{1}{T^{\delta/2}} Y_1^T(Tt)$, $t \geq 0$, we see that $(Y^{(T)}_1(t))_{t \geq 0}$ is a martingale with quadratic variation
\[ \langle Y^{(T)}_1, Y^{(T)}_1 \rangle_t = \frac{2}{T^\delta} \int_0^{Tt} \frac{H^{(T)}(Y^T_2(s))}{V(Y_1^T(s))} \, ds = 2 \, D^{V,T}(t) \]
which converges in law, when $T \to \infty$, to $ (2 \,\Delta(t))_{t \geq 0}$. An important fact in the sequel will be the fact that the martingale property is preserved when we take a limit in law. This will be guaranteed by the following lemma:

\begin{lemma}
\label{lem:ui}
For all $t>0$ and a.e. trajectory of $(\mathcal{H}(x))_{x \in \mathbb{R}}$, we have 
$$\sup_{T \geq 1} E^\omega_0[(D^T(t))^2] < \infty,$$ 
in particular, the family of random variables $(Y^T_1(t)^2)_{T \geq 1}$ is uniformly integrable with respect to $P^\omega_0$. 
\end{lemma}

\begin{proof}
The uniform integrability claim follows from the estimate. Indeed, by the BDG inequality
$$E^\omega_0[Y^T_1(t)^4] \leq C E^\omega_0 [\langle Y^T_1, Y^T_1 \rangle_t^2] = 4 C E^\omega_0[(D^{V,T}(t))^2],$$
uniformly over $T \geq 1$, where $C>0$ is the BDG constant for $p=4$. In turn, by the first condition of Assumption \ref{assump:H_V_tail_bis}, we have $D^{V,T}(t) \leq \frac{1}{c} D^T(t)$, whence the claimed estimate will yield $\sup_{T \geq 1} E^\omega_0[Y^T_1(t)^4] < \infty$, thereby providing uniform integrability of $(Y^T_1(t)^2)_{T \geq 1}$. There remains to prove the uniform estimate on $E^\omega_0[(D^T(t))^2]$. To alleviate notations, let us assume that $t=1$, the proof for other values of $t$ is similar. Note that 
$$(D^T(1))^2 = \int_{\mathbb{R}} \int_{\mathbb{R}} \tilde{\ell}^{Y,T}_1(x) \, \tilde{\ell}^{Y,T}_1(y) \, d \mu^T(x) d \mu^T(y).$$
Since $\tilde{\ell}^{Y,T}_1(x) = \ell^{T}_1(x) $ for all $x \in \mathbb{Z}/ \sqrt{T}$, and since $\ell^T_1(x) = 0$ unless $\sqrt{T} |x| \leq \max_{[0,T]} |Y_2|$, we can rewrite the above as
$$\int_{\mathbb{R}} \int_{\mathbb{R}} \ell^{Y,T}_1(x) \, \mathbf{1}_{\{\max_{[0,T]} |Y_2| \geq \sqrt{T} |x|\}} \, \ell^{Y,T}_t(y) \, \mathbf{1}_{\{\max_{[0,T]} |Y_2| \geq |y|\}} d \mu^T(x) d \mu^T(y). $$
Taking expectations and applying Cauchy-Schwarz shows that $E^\omega_0[(D^T(1))^2]$ is bounded by
$$\int_{\mathbb{R}^2} E^\omega_0[\ell^{Y,T}_1(x)^4]^{\frac{1}{4}} \, E^\omega_0[\ell^{Y,T}_1(y)^4]^{\frac{1}{4}} \, \mathbb{P} (\max_{[0,T]} |Y_2| \geq \sqrt{T} |x|) \, \mathbb{P}(\max_{[0,T]} |Y_2| \geq \sqrt{T} |y|) \, d \mu^T(x) \, d \mu^T(y). $$ 
But, by \eqref{eq.bound.lp.lt} (see Subsection \ref{subsec:appendix_local_times} of the Appendix), $E^\omega_0[\ell^{Y,T}_1(x)^4] \leq c(2)$ where $c(2)>0$ is a constant. In addition, fixing $k \in \mathbb{N}$ sufficiently large (in a way that will be specified later), there exists $c>0$ depending only on $k$ such that, for all $x \in \mathbb{R}$ such that $|x| \geq 1$, 
\begin{equation}
\label{eq:bound_deviation_maxY2}
\mathbb{P}(\max_{[0,T]} |Y_2| \geq \sqrt{T} |x|) \leq c \, |x|^{-2k}.
\end{equation}
Indeed, since $Y_2$ is a simple random walk, we can represent it as $Y_2(t) = S_{N_t}$, for $(S_k)_{k \in \mathbb{N}}$ a discrete-time simple random walk, and $(N_t)_{t \geq 0}$ a standard Poisson process independent of $Y_2$.  Now, by equation (12.11) in \cite{lawler2010random}, we have, for all $n \geq 1$, and $\lambda >0$, 
$$\mathbb{P}(\max_{0 \leq j \leq n} |S_j| \geq  \lambda \sqrt{n}) \leq c \lambda^{-2k}.$$ 
Hence,
\[\begin{split}
\mathbb{P}(\max_{[0,T]} |Y_2| \geq \sqrt{T} |x|)  &= \sum_{n=0}^\infty \mathbb{P} \left( N_T=n , \max_{0 \leq j \leq n} |S_j| \geq  |x| \sqrt{\frac{T}{n}} \sqrt{n} \right) \\
&\leq  c T^{-k} \sum_{n=0}^\infty \mathbb{P}(N_T=n) \, n^k  |x|^{-2k} = c T^{-k} \, \mathbb{E}[N_T^k] \, |x|^{-2k}.
\end{split}\]
By the bound of \cite{AHLE2022109306}, we have $\mathbb{E}[N_T^k] \leq c_k T^k$, where $c_k = e^{\frac{k^2}{2}}$ (note that $T \geq 1$). The claimed bound \eqref{eq:bound_deviation_maxY2} follows. Thus we get
$$E^\omega_0[(D^T(1))^2] \leq C \left(\int_\mathbb{R} (1 \wedge |x|^{-2k}) \, d \mu_T(x) \right)^2, $$
for some constant $C$ depending only on $k$.  Note that, for any $m > \frac{1}{\alpha_1}$,  by Theorem 13 of \cite{bertoin1996levy} with $\Phi(\lambda) = \lambda^\alpha_1$ and $h(t) = t^{m}$, we have $\mathcal{H}(x) \underset{|x| \to \infty}{=} o(x^{m})$ almost-surely.  
Hence let us choose $k$ such that $2k > \frac{1}{\alpha_1}$. As $T \to \infty$, $\int_{\mathbb{R}} (1 \wedge |x|^{-2k}) \, d \mu_T(x)$ converges to
\begin{equation}
\label{eq:continu_limit_integral}
2k \int_{\mathbb{R}} \mathbf{1}_{\{|x| \geq 1\}} |x|^{-2k-1} \, \mathcal{H}(x) \, dx, 
\end{equation}
as can be verified upon writing 
$$\int_{\mathbb{R}} (1 \wedge |x|^{-2k}) \, d \mu_T(x) =  \sum_{k \in \mathbb{Z}} (1 \wedge |k T^{-1/2}|^{-2k}) \left(\mathcal{H}((k+1)T^{-1/2}) -\mathcal{H}(k T^{-1/2}) \right) $$
and performing a discrete summation by parts. Since $2k > \frac{1}{\alpha_1}$, the integral in \eqref{eq:continu_limit_integral} is finite, and hence $\sup_{T \geq 1} E^\omega_0[(D^T(1))^2] < \infty$.
\end{proof}


\subsubsection{Martingale convergence using a Girsanov argument}

We can prove the following scaling limit result for the time-changed random walk $(Y(t))_{t \geq 0}$.

\begin{proposition}
\label{thm:conv_y}
For $\mathbb{P}$ a.e. environment, when $T \to \infty$, the process $(Y^{(T)}(t))_{t \geq 0}$ converges in law, in $D([0,\infty), \mathbb{R})^2$, to the process $(\mathcal{Y}_1(t), \mathcal{Y}_2(t))_{t \geq 0}$, where 
\begin{equation}
\label{eq:expression_limit_process}
\mathcal{Y}_1(t) = B_1(\Delta(t)), \quad \mathcal{Y}_2(t) = B_2(t), \qquad t \geq 0,
\end{equation}
where $B_1$ and $B_2$ are two independent Brownian motions with diffusion coefficient $\sqrt{2}$ and $\Delta(t) = \int_{\mathbb{R}} L^{B_2}(t,x) \, \mathcal{H}(dx)$, $t \geq 0$.
\end{proposition}

\begin{proof}
Note that although the processes $Y^{(T)}_1$ and  $Y^{(T)}_2$ are defined with different scaling exponents, they are adapted to the same filtration
\[\sigma\{Y(Ts), \, s \leq t\}, \qquad t \geq 0. \]
They are actually martingales for the above filtration, their respective quadratic variations are given by 
\[ \langle Y^{(T)}_1, Y^{(T)}_1 \rangle_t =  2 D^{V,T}(t), \quad \langle Y^{(T)}_2, Y^{(T)}_2 \rangle_t  = 2 t, \quad \langle Y^{(T)}_1, Y^{(T)}_2 \rangle_t = 0 \qquad t \geq 0, \]
where we recall that $(D^{V,T}(t))_{t \geq 0}$ converges in law, as $T \to \infty$, to the Kesten-Spitzer process $(\Delta(t))_{t \geq 0}$. Note further that
$Y^{(T)}_2 = \frac{1}{\sqrt{T}}Y_2(Tt)$ where $Y_2$ is a simple random walk, while we can write $Y^{(T)}_1 = S^{T^\delta}_1 \circ D^T$, where 
$$S^{T^\delta}_1(t) = \frac{1}{T^{\delta/2}} S_1(T^\delta t),  \quad t \geq 0,$$
where $S_1$ is a one-dimensional simple random walk. By Donsker's theorem $(Y^{(T)}_2)_{T >0}$ and $(S^{T^\delta}_1)_{T >0}$ both converge in law, in $D([0,\infty))$, to the law of a Brownian motion; in particular, these sequences are both tight in $D([0,\infty))$. By Corollary \ref{cor:comp_tightness}, $(Y^{(T)}_1)_{T>0}$ is therefore tight as well. So the sequence of processes $(Y^{(T)}_1, Y^{(T)}_2)$ is tight in $D([0,\infty))^2$. Let us denote by $\mathcal{Y}(t):=(\mathcal{Y}_1(t),\mathcal{Y}_2(t))$, $t \geq 0$, a sub-sequential limit thereof, whose law we set out to determine.
By the above, for all $T>0$, the processes
\[Y^{(T)}_1(t), \quad Y^{(T)}_2(t),  \quad Y^{(T)}_1(t)^2 - 2 D^{V,T}(t), \quad Y^{(T)}_1(t) Y^{(T)}_2(t), \quad \hbox{and} \quad Y^{(T)}_2(t) - 2t \]
are martingales with respect to the filtration 
\[\mathcal{F}^T_t := \sigma\{\mathcal{Y}^T(s), \, s \leq t\}, \qquad t \geq 0. \]
Thanks to the uniform integrability ensured by Lemma \ref{lem:ui}, sending $T \to \infty$, we deduce that
\[\mathcal{Y}_1(t), \quad \mathcal{Y}_2(t),  \quad \mathcal{Y}_1(t)^2 - 2 \Delta(t), \quad \mathcal{Y}_1(t) \mathcal{Y}_2(t), \quad \hbox{and} \quad \mathcal{Y}_2(t) - 2t \]
are martingales with respect to the filtration 
\[\mathcal{F}_t := \sigma\{\mathcal{Y}(s), \, s \leq t\}, \qquad t \geq 0. \]
Thus $(\mathcal{Y}_t)_{t \geq 0}$ is a \textit{continuous} martingale with respect to the filtration $(\mathcal{F}_t)_{t \geq 0}$ with quadratic variations given by 
\[ \langle \mathcal{Y}_1, \mathcal{Y}_1 \rangle_t =  2\Delta_t, \quad \langle \mathcal{Y}_2, \mathcal{Y}_2 \rangle_t  = 2t, \quad \langle \mathcal{Y}_1, \mathcal{Y}_2 \rangle_t = 0, \qquad t \geq 0. \]
Moreover, by the convergence result of the previous section, $\Delta$ is expressed in terms of $\mathcal{Y}_2$ as $\Delta(t) = \int_{\mathbb{R}} L^{\mathcal{Y}_2}(t,x) \, \mathcal{H}(dx)$. Now, for all $s \geq 0$, $\Delta^{-1}(s)$ is an $(\mathcal{F}_t)_{t \geq 0}$  stopping time.  By the optional stopping theorem (see Theorem (3.2) of Chapter II in \cite{revuz2013continuous}), we  deduce that, for all $n \geq 1$ finite and $r < s$, 
\[ E^\omega_0[\mathcal{Y}_i (\Delta^{-1}(s) \wedge n) | \mathcal{F}_{\Delta^{-1}(r)} ] =  \mathcal{Y}_i (\Delta^{-1}(r) \wedge n), \qquad i =1,2, \]
that is
\[ E^\omega_0[\mathcal{Z}_i (s \wedge \Delta(n)) | \mathcal{G}_r ] =  \mathcal{Z}_i (r \wedge \Delta(n)), \qquad i =1,2. \]
where $\mathcal{Z}(s) := \mathcal{Y}(\Delta^{-1}(s))$, $s \geq 0$, and where we introduced the filtration $\mathcal{G}_u := \mathcal{F}_{\Delta^{-1}(u)}$, $u \geq 0$. Since $\Delta(n) \underset{n \to \infty}{\longrightarrow} \infty$ a.s. (by Lemma \ref{lem:delta_increasing}), we deduce that $(\mathcal{Z}(s))_{s \geq 0}$ is a $(\mathcal{G}_u)_{u \geq 0}$ local martingale, with quadratic variations given by
\[ \langle \mathcal{Z}_1, \mathcal{Z}_1 \rangle_s = 2 s, \quad \langle \mathcal{Z}_2, \mathcal{Z}_2 \rangle_s  = 2 \Delta^{-1}(s), \quad \langle \mathcal{Z}_1, \mathcal{Z}_2 \rangle_s = 0 \qquad s \geq 0. \]
In particular, $\mathcal{Z}_1 (s)  = B_1(s)$, where $B_1$ is a $(\mathcal{G}_s)_{s \geq 0}$ Brownian motion with diffusion coefficient $\sqrt{2}$. Let us now characterize the joint law of $(\mathcal{Z}_1, \mathcal{Z}_2)$. Let $\alpha \in L^2(\mathbb{R}_+)$ and let $\Phi$ be a continuous, bounded function on $C([0,\infty))^2$. By Girsanov's theorem
\[\begin{split}
E^\omega_0 \left[\exp \left(\int_0^{+\infty} \alpha_s dB_1(s) \right) \Phi(\mathcal{Z}_2, \Delta)\right] &= \exp \left(\frac{1}{2} \int_0^{+\infty} \alpha_s^2 ds \right) \tilde{E}^\omega_0 \left[ \Phi(\mathcal{Z}_2, \Delta)\right]
\end{split}\]
where $\frac{d \tilde{P}^\omega_0}{d P^\omega_0} \biggr \rvert_{\mathcal{G}_t} =  \exp \left( \int_0^{+\infty} \alpha_s dB_1(s) - \frac{1}{2} \int_0^{+\infty} \alpha_s^2 ds \right)$. Now recall that 
\[\forall s \geq 0, \qquad \langle B_1, \mathcal{Z}_2\rangle_s = \langle \mathcal{Z}_1, \mathcal{Z}_2 \rangle_s = 0. \]
Hence under $ \tilde{P}^\omega_0$,  $(\mathcal{Z}_2(s))_{t \geq 0}$ is still a $(\mathcal{G}_s)_{s \geq 0}$ local martingale, with quadratic variation given by $2 \Delta^{-1}(s)$. Therefore, by another application of the optional stopping theorem as above, and using the fact that $\Delta^{-1}(n) \underset{n \to \infty}{\longrightarrow} \infty$ (see Lemma \ref{lem:delta_increasing}) , we deduce that the process $\mathcal{Z}_2(\Delta(t)) = \mathcal{Y}_2(t)$ is a continuous local martingale under $\tilde{P}^\omega_0$ with respect to the filtration $(\mathcal{F}_t)_{t \geq 0}$, and with quadratic variation $\langle \mathcal{Y}_2, \mathcal{Y}_2 \rangle_t = 2 t$. By L\'{e}vy's characterization theorem, under $\tilde{P}^\omega_0$, we still have $\mathcal{Y}_2(t) = B_2(t)$ where $B_2$ is an $(\mathcal{F}_t)_{t \geq 0}$ Brownian motion with diffusion coefficient $\sqrt{2}$, so that the joint law of the processes $\Delta(t) = \int_{\mathbb{R}} L^{\mathcal{Y}_2}_t(x) \, \mathcal{H}(dx)$ and of $\mathcal{Z}_2(t) = \mathcal{Y}_2 (\Delta^{-1}(t))$, $t \geq 0$, remains unchanged. We therefore get
\[\begin{split}
E^\omega_0 \left[\exp \left(\int_0^{+\infty} \alpha_s dB_1(s) \right) \Phi(\mathcal{Z}_2, \Delta)\right] &= \exp \left(\frac{1}{2} \int_0^{+\infty} \alpha_s^2 ds \right) E^\omega_0 \left[\Phi(\mathcal{Z}_2,\Delta)\right].
\end{split}\]
This shows that $B_1$ is a Brownian motion independent of $(\mathcal{Z}_2, \Delta)$, in particular, $B_1$ is independent of $\mathcal{Z}_2 \circ \Delta = \mathcal{Y}_2$. We thus deduce that 
\[\mathcal{Y}_1(t) = \sqrt{2} B_1(\Delta(t)), \quad \mathcal{Y}_2(t) = \sqrt{2} B_2(t), \qquad t \geq 0,\]
where $B_1$ and $B_2$ are two independent Brownian motions with diffusion coefficient $\sqrt{2}$ and $\Delta(t) = \int_{\mathbb{R}} L^{B_2}(t,x) \, \mathcal{H}(dx)$, $t \geq 0$. The uniqueness of the limit, and thereby the claim, is proved. \end{proof}

\subsubsection{Conclusion for the original model}

Recall that the process $(Y(t))_{t \geq 0}$ is related to the original line model $(X(t))_{t \geq 0}$ via the relation \eqref{eq:timechange_x_to_y}. Thanks to Proposition \ref{thm:conv_y}, we can finally complete the proof of Theorem \ref{thm:conv}.


\begin{proof}[Proof of Theorem \ref{thm:conv}]
From the relation \eqref{eq:timechange_x_to_y}, we have, for all $T>0$, 
\[\int_0^t H(X^T_2(s)) \,ds  = \int_0^{A(t)} \frac{H(Y^T_2(r))}{V(Y^T_1(r))} \, dr, \qquad \hbox{where} \qquad A(t) = \int_0^t V(X^T_1(s)) ds, \quad t \geq 0.\]
In the first equality we performed the change of variable $r = A(s)$ to obtain the right-hand side. It therefrom follows that, setting 
\begin{equation}
\label{eq:def.aht1}
A^{H,T}_1(t) := \frac{1}{T^\delta} \int_0^{Tt} H(X_2^T(s)) \,ds, 
\end{equation}
then $A^{H,T}_1 = D^{V,T} \circ A^{(T)}$, where $D^{V,T}$ was defined in \eqref{eq:def_dvt}, and where, for all $t \geq 0$, $A^{(T)}(t) := \frac{1}{T} A(Tt)$. Now we have proven above that $D^{V,T} \underset{T \to \infty}{\Longrightarrow} \Delta$ in $C$. In addition we have the convergence of processes
\[ A^{(T)}(t) := \frac{1}{T} \int_0^{Tt} V(X^T_1(s)) \, ds \underset{T \to \infty}{\longrightarrow} t\]
a.s. in $C$. Indeed, since $\mathbb{E}[V(0)]=1$, by the ergodicity of the environment seen from the particle, we have, for all fixed $t \geq 0$, the almost-sure convergence
\[A^{(T)}(t) \underset{T \to \infty}{\longrightarrow} t,\]
whence, reasoning as in the proof of Proposition \ref{prop:over_scaling}, the claimed convergence follows.
Hence, by Slutsky's lemma, we have the convergence in distribution
\[ (D^{V,T} ,A^{(T)}) \underset{T \to \infty}{\Longrightarrow} (\Delta, I) \]
in $C \times C$, where $I$ denotes the identity process $I(t)=t$, $t \geq 0$. By Lemma \ref{lem:composition}, we deduce that $D^{V,T} \circ A^{(T)} \underset{T \to \infty}{\Longrightarrow} \Delta$ in $C$, which yields \eqref{eq:conv.clock_proc_delta}. 
We now prove the convergence statement \eqref{eq:conv_x_statement}. We have 
\[X^{(T)}_1(t) = \frac{1}{T^{{\delta/2}}} Y_1 \left(T A^{(T)}(t)\right), \quad  X^{(T)}_2(t) = \frac{1}{\sqrt{T}} Y_2 \left(T A^{(T)}(t)\right), \qquad t \geq 0,\] 
that is $X^{(T)}(t) = Y^{(T)} \circ A^{(T)} (t)$, $t \geq 0$. But, by Proposition \ref{thm:conv_y}, we have
$Y^{(T)} \underset{T \to \infty}{\Longrightarrow} \mathcal{Y}$ in $D$, where $\mathcal{Y}$ is given by \eqref{eq:expression_limit_process}. 
By another application of Slutsky's lemma, we obtain the convergence in distribution
\[ (Y^{(T)} ,A^{(T)}) \underset{T \to \infty}{\longrightarrow} (\mathcal{Y}, I) \]
in $D \times C$, where $I(t) =t$, $t \geq 0$. Since $\mathcal{Y}$ admits continuous paths, by Lemma \ref{lem:composition}, the claim \eqref{eq:conv_x_statement} follows.
\end{proof}

\begin{remark}
As a by-product of the above proofs  we actually have the convergence in law in $D \times D \times C$
\begin{equation}
\label{eq:conv_x_delta}
(X^{(T)}_1, X^{(T)}_2, A^{H,T}_1) \underset{T \to \infty}{\Longrightarrow} (B_1 \circ \Delta, B_2, \Delta). 
\end{equation}
where $A^{H,T}_1$ is as in \eqref{eq:def.aht1}, and where $B_1,B_2, \Delta$ are as in the statement of Theorem \ref{thm:conv}. This enhanced convergence will be useful in the next section.
\end{remark}

\subsubsection{Scaling limit of the Constant Speed Random Walk}

Let us consider now the Constant Speed Random Walk $(\hat{X}(t))_{t \geq 0}$, whose generator is given by
\[\begin{split}
\hat{\mathcal{L}}^\omega f(x) &= \frac{H(x_2)}{V(x_1) + H(x_2)}(f (x + e_1) - 2 f(x) + (f (x - e_1)) \\ 
&+ \frac{V(x_1)}{V(x_1) + H(x_2)}(f (x + e_2) -  2f(x) + f(x-e_2))  
\end{split}\]
where $H(k)$, $V(k)$, $k \in \mathbb{Z}$ satisfy Assumption \ref{assump:H_V_tail_bis}.
Note that $(\hat{X}(t))_{t \geq 0}$ can be represented using $(X(t))_{t \geq 0}$ via a time-change, namely, $\hat{X}(t) \overset{(d)}{=} X(J^{-1}(t))$, where
\begin{equation}
\label{eq:time_change_csrw}
J^{-1}(t) := \inf \{ r \geq 0: J(u) \geq t\}, \qquad  J(r) = \int_0^r (V(X_1(s)) + H(X_2(s))) \, ds.  
\end{equation}
We also define for all $T>0$ the process $\hat{X}^T$, just replacing the variables $H(x_2)$ in the generator above by $H^{(T)}(x_2)$ (recall definition \eqref{eq:def_env_rescaled}). Recalling that $\delta := \frac{1}{2}\left(1 + \frac{1}{\alpha_1}\right)$, we will prove the following result:

\begin{theorem}
\label{thm:limit_csrw}
For $T>0$, let $\hat{X}^{(T)}(t) := (\hat{X}^{(T)}_1(t), \hat{X}^{(T)}_2(t))$, $t \geq 0$, where 
\[ \hat{X}^{(T)}_1(t) := \frac{1}{\sqrt{T}} \hat{X}^T_1(Tt), \quad \hat{X}^{(T)}_2(t) := T^{-\frac{1}{2\delta}} \hat{X}^T_2(Tt), \qquad t \geq 0. \]
Then, for $\mathbb{P}$-a.e. $\omega$, as $T \to \infty$, the process $(\hat{X}^{(T)}(t))_{t \geq 0}$ converges in distribution, in $D \times D$, to  
$\mathcal{X}(t) = (\mathcal{X}_1(t), \mathcal{X}_2(t))_{t \geq 0}$, where 
\begin{equation}
\label{eq:expression_limit_process_csrw}
\mathcal{X}_1(t) = B_1(t), \quad \mathcal{X}_2(t) = B_2( \Delta^{-1}(t)), \qquad t \geq 0,
\end{equation}
where $B_1$ and $B_2$ are two independent Brownian motions with diffusion coefficient $\sqrt{2}$,  $\Delta(t) = \int_{\mathbb{R}} L^{B_2}_t(x) \, \mathcal{H}(dx)$, $t \geq 0$, and $(\Delta^{-1}(t))_{t \geq 0}$ denotes the inverse of $(\Delta(t))_{t \geq 0}$ (note that, by Lemma \ref{lem:delta_increasing}, the paths of $(\Delta(t))_{t \geq 0}$ are strictly increasing).
\end{theorem}

\begin{remark}
The process $B_2( \Delta^{-1}(t))$, $t \geq 0$, is the FIN process considered in \cite{FIN}. 
\end{remark}

\begin{proof}
We have
\begin{equation}
\label{eq:rep_z}
\hat{X}^{(T)} \overset{(d)}{=} X^{(T^{1/\delta})} \circ {J^{(T^{1/\delta})}}^{-1},  
\end{equation}
where, recalling definition \eqref{eq:time_change_csrw}, we define for all $M>0$
\[J^{(M)}(t) := \frac{1}{M^\delta} \int_0^{M t} (V(X_1^T(s)) + H^{(T)}(X_2^T(s))) \, ds, \qquad t \geq 0, \]
and ${J^{(M)}}^{-1}$ is the inverse of $J^{(M)}$, given by
\[{J^{(M)}}^{-1}(t) =  \frac{1}{M} J^{-1}(M^\delta t), \qquad t \geq 0. \]
Now note that 
\[J^{(M)}(t) = A^{H,M}_1(t) + \frac{1}{M^{\delta}} \int_0^{Mt} V(X_1^T(s)) \, ds, \qquad t \geq 0 \]
where $A^{H,M}_1(t)$ was defined in \eqref{eq:def.aht1}. By \eqref{eq:conv.clock_proc_delta}, $A^{H,M}_1(t) \underset{T \to \infty}{\Longrightarrow} \Delta$ in $D$; moreover, since $\delta>1$, there holds the convergence of processes, in $C$,  
\[\frac{1}{M^{\delta}} \int_0^{Mt} V(X^T_1(s)) \, ds \underset{M \to \infty}{\longrightarrow} 0 \qquad \hbox{a.s.} \]
By \eqref{eq:conv_x_delta} and by Slutsky's Lemma, we deduce the convergence in law in $D \times D  \times C$
\[(X^{(T^{1/\delta})}_1, X^{(T^{1/\delta})}_2, J^{(T^{1/\delta})}) \underset{T \to \infty}{\Longrightarrow} (B_1 \circ \Delta, B_2, \Delta). \]
By the analytical Lemma \ref{lem:conv.inv}, and the fact that, a.s., the process $(\Delta(t))_{t \geq 0}$ is strictly increasing, we obtain in turn the convergence in law 
\[(X^{(T^{1/\delta})}_1, X^{(T^{1/\delta})}_2, {J^{(T^{1/\delta})}}^{-1}) \underset{T \to \infty}{\Longrightarrow} (B_1 \circ \Delta, B_2, \Delta^{-1}), \]
again in the space $D \times D \times C$. Finally, by \eqref{eq:rep_z} and using the remark following Lemma \ref{lem:composition}, we obtain
the convergence in law in $D \times D$
\[ \hat{X}^{(T)}  \underset{T \to \infty}{\Longrightarrow} ((B_1 \circ \Delta) \circ \Delta^{-1} , B_2 \circ \Delta^{-1}), \]
which yields the claimed result.
\end{proof}

\section{Appendix}

\subsection{Definitions, and useful properties of stable processes}

We recall the following definitions which can be found, e.g., in Section 1 of \cite{Bertoin1999}.

\begin{definition}
A subordinator $(Z_t)_{t \geq 0}$ is a right-continous, non-decreasing process, with stationary, independent increments. For such a process there exists $\Phi: \mathbb{R}_+ \to \mathbb{R}_+$, called the Laplace exponent of $Z$, such that
\[\forall \lambda \geq 0, \quad \forall t \geq 0, \qquad \mathbb{E}[e^{-\lambda Z_t}] = e^{-t \Phi(\lambda)}. \]
\end{definition}

\begin{definition}
Let $\alpha \in (0,1)$. A stable subordinator $(Z_t)_{t \geq 0}$ of index $\alpha$ is a subordinator with Laplace exponent
$$\Phi(\lambda) = \lambda^\alpha, \qquad \lambda \geq 0. $$
We call double-sided stable subordinator of index $\alpha$ the process $(Z(x))_{x \in \mathbb{R}}$ where 
\[Z(x) = \begin{cases}
Z_+(x), \quad &x \geq 0 \\
Z_{-}(-x), \quad &x < 0
\end{cases}\]
where $Z_+$ and $Z_{-}$ are two independent stable subordinators of index $\alpha$.
\end{definition}

Finally we state a useful fact about Kesten-Spitzer processes. Let $\alpha \in (0,1)$,
and $(Z(x))_{x \in \mathbb{R}}$ be a double-sided subordinator of index $\alpha$, let 
$(B(t))_{t \geq 0}$ be a Brownian motion independent of $Z$, and consider 
the Kesten-Spitzer process $(\Delta(t))_{t \geq 0}$ given by
\begin{equation}
\label{eq:exp_ks_process_appendix}
\Delta(t) := \int L^B(t,x) \, \mathcal{H}(dx), \qquad t \geq 0.
\end{equation}

\begin{lemma}
\label{lem:delta_increasing}
Almost-surely, $(\Delta(t))_{t \geq 0}$ is strictly increasing on $\mathbb{R}_+$, $\Delta(t) \underset{t \to \infty}{\longrightarrow} \infty$, and $\Delta^{-1}(s) \underset{s \to \infty}{\longrightarrow} \infty$.
\end{lemma}

\begin{proof}
Since, for all $x \in \mathbb{R}$, $t \mapsto L^B(t,x)$ is non-decreasing, $t \mapsto \Delta(t)$ is non-decreasing. To prove that it is a.s. strictly increasing, it suffices to show that, for all $s<t$, $\Delta(s)  < \Delta(t)$ a.s.: indeed this will imply that, a.s., $(\Delta(t))_{t \geq 0}$ is strictly increasing on $\mathbb{Q}_+$, which, by density of the latter in $\mathbb{R}_+$, will yield the claim. Since the increments of $(\Delta(t))_{t \geq 0}$ are stationary (see \cite{kesten1979limit}), and by scale-invariance of this process, it suffices in turn to show that $\Delta(1) >0$ a.s. By scale-invariance this is equivalent to checking that $\Delta(t) \underset{t \to \infty}{\longrightarrow} + \infty$ a.s. Indeed, we have
\[\mathbb{P}(\Delta(1)=0) = \lim_{\lambda \to + \infty} \mathbb{E}[e^{-\lambda \Delta(1)}] =   
\lim_{\lambda \to +\infty} \mathbb{E}[e^{-\Delta(\lambda^{1/\delta})}],\]
where $\delta = \frac{1}{2}\left(1 + \frac{1}{\alpha}\right)$ is the Kesten-Spitzer scaling exponent. But since the local times of the Brownian motion
$B$ satisfy $L^B(t,x) \underset{t \to + \infty}{\longrightarrow} + \infty$ a.s. for all $x \in \mathbb{R}$, the monotone convergence theorem applied
to the integral expression $\eqref{eq:exp_ks_process_appendix}$ yields $\Delta(t) \underset{t \to \infty}{\longrightarrow} + \infty$ a.s. This 
yields at once the first two statements of the claim. The last statement follows immediately from the definition of $\Delta^{-1}$ as the inverse of a strictly increasing function on $\mathbb{R}_+$.
\end{proof}

\subsection{A useful estimate for maxima of heavy-tailed random variables}

We recall the following result which follows from estimates in \cite{Andres:2016aa}, see equation (A.1) therein.

\begin{lemma}
\label{lem:estimate_max_heavy_rv}
Let $\alpha \in (0,1)$ and $Z_k$, $k \in \mathbb{Z}$, be a collection of i.i.d. random variables defined on the same probability space $(\Omega, \mathcal{F}, \mathbb{P})$ with tail behavior  $\mathbb{P}(Z_k > t) = t^{-\alpha + o(1)}$ as $t \to \infty$. Let $\alpha^{-} \in (0, \alpha) $. Then, for $\mathbb{P}$-a.e. $\omega$, there exists $C(\omega)>0$, such that, for all $k \geq 1$,
\[\max_{-k \leq i \leq k} Z_i \leq C(\omega) (1+ k^{\frac{1}{\alpha^{-}}}). \] 
\end{lemma}

\begin{proof}
By equation (A.1) in \cite{Andres:2016aa}, where we arbitrarily choose $\delta = 0.5$, for $\mathbb{P}$-a.e. $\omega$, there exists $L_0 = L_0(\omega)$ such that, for all $k \geq L_0$,
\[\max_{-k \leq i \leq k} Z_i \leq (k (\log k)^{1,5})^{1/\alpha}. \]
But the right-hand side is bounded for $k$ sufficiently large by $1 + k^{\frac{1}{\alpha^{-}}}$, and the claim follows.
\end{proof}

\subsection{Estimates on solutions of a system of coupled non-linear differential inequalities}  
\label{app:proof_estimates_system}

Here we state and prove Lemma \ref{lem:bound_sol_diff_ineq} which was used in the proof of Lemma \ref{lem:Xstar_is_conservative} to bound moments of the time-changed random walk

\begin{lemma}
\label{lem:bound_sol_diff_ineq}
Let $A_i :=  \frac{1+ \varepsilon_i^+}{1- \varepsilon_1^+ \varepsilon_2^+}$, $i=1,2$. Let $C,\kappa >0 $, and let $h,v : \mathbb{R}_+ \to \mathbb{R}_+$ be two continuous functions.
\begin{enumerate}
\item If $h$ and $v$ satisfy the system of differential inequalities
\begin{equation}
\label{eq:system_diff_ineq}
\forall t \in  [0,T], \quad \begin{cases}
h(t) \leq C(\kappa + A_1  \int_0^{t}  v(s)^{\varepsilon_1^+} \, ds) \\
v(t) \leq C(\kappa +  A_2 \int_0^{t}  h(s)^{\varepsilon_2^+} \, ds), 
\end{cases}
\end{equation}
then, for all $t \in [0,T]$,
$$h(t) \leq [(C\kappa)^{1/A_1}+Ct]^{A_1}, \qquad v(t) \leq [(C\kappa)^{1/A_2}+Ct]^{A_2}. $$ 

\item If $h$ and $v$ satisfy the system of differential inequalities
\begin{equation}
\label{eq:system_diff_ineq_bis}
\forall t \in  [0,T], \quad \begin{cases}
h(t) \leq C(\kappa + A_1  \int_0^{t}  v(s)^{\varepsilon_1^+ \varepsilon_2^+} \, ds) \\
v(t) \leq C(\kappa + \frac{A_2}{\varepsilon_2^+} \int_0^{t}  h(s)^{\varepsilon_2^+} v(s)^{1-\varepsilon_2^+} \, ds),
\end{cases}
\end{equation}
then, for all $t \in [0,T]$,
$$h(t) \leq [(C\kappa)^{1/A_1} + C t]^{A_1}, \qquad v(t) \leq [(C\kappa)^{1/B_2} +Ct]^{B_2}, $$
where $B_2=\frac{A_2}{{\varepsilon_2}^+}$.
\end{enumerate} 
\end{lemma}

\begin{proof}
We start by proving the first assertion. Let $C^{+}>C$, and let 
\[x(t) = [(C^{+} \kappa)^{1/A_1}+ C^{+}t]^{A_1}, \quad y(t) = [ (C^{+} \kappa)^{1/A_2} + C^{+}t]^{A_2}, \qquad t \geq 0.\]
Note that $x$ and $y$ solve the system 
\begin{equation}
\label{eq:system_diff_eq_xy}
\forall t \in  [0,T], \quad \begin{cases}
x(t) = C^{+} (\kappa + A_1  \int_0^{t}  (y(s))^{\varepsilon_1^+} \, ds) \\
y(t) = C^{+} (\kappa +  A_2 \int_0^{t}  (x(s))^{\varepsilon_2^+} \, ds). 
\end{cases}
\end{equation}
Let further 
\[ I :=  \{ t \in [0,T] : \forall s \in [0,t], \quad  h(s) < x(s) \quad \hbox{and} \quad v(s) < y(s)\}. \]
By continuity of $h,v,x$ and $y$,  $I$ is an open interval. Furthermore we have $0 \in I$. Let us assume by contradiction
that $\sup I  < T$, and let us denote by $t$ this supremum. Then, by continuity of  $h,v,x$ and $y$, we either have $h(t) = x(t)$ or $v(t)=y(t)$, let us assume WLOG that $h(t)=x(t)$. We necessarily have $t>0$, but by assumption
$$h(t) \leq C(\kappa + A_1  \int_0^{t}  (v(s))^{\varepsilon_1^+} \, ds). $$
Since, by definition of $t$, we have $v(s)  < y(s)$ for all $s < t$, we obtain 
$$h(t) < C(\kappa + A_1  \int_0^{t}  (y(s))^{\varepsilon_1^+} \, ds) < C^{+}(\kappa + A_1  \int_0^{t}  (y(s))^{\varepsilon_1^+} \, ds) = x(t) $$
so $h(t) < x(t)$ which is a contradiction.  Hence $\sup I = T$, whence we deduce that 
$$\forall t \in [0,T], \quad h(t) <  [(C^{+}\kappa)^{1/A_1}+ C^{+}t]^{A_1} \quad \hbox{and} \quad v(s) <  [(C^{+} \kappa)^{1/A_2}+C^{+}t]^{A_2}, $$
which, since $C^{+}>0$ can be chosen arbitrarily close to $C$, yields the first claim. 
For the second statement, we proceed in exactly the same way as for the first one, noting now that, for all $\epsilon>0$, the functions 
\[x(t) = [(C^{+} \kappa)^{1/A_1}+ C^{+} t)]^{A_1}, \quad y(t) = [(C^{+} \kappa)^{1/B_2}+C^{+}t]^{B_2}, \qquad t \geq 0,\]
with $B_2=\frac{A_2}{{\varepsilon_2}^+}$, will solve the system 
\begin{equation}
\label{eq:system_diff_eq_xy_bis}
\forall t \in  [0,T], \quad \begin{cases}
x(t) = C^{+}(\kappa + A_1  \int_0^{t}  y(s)^{\varepsilon_1^+ \varepsilon_2^+} \, ds) \\
y(t) = C^{+}(\kappa +  B_2 \int_0^{t}  x(s)^{\varepsilon_2^+}  y(s)^{1 - \varepsilon_2^+}\, ds). 
\end{cases}
\end{equation}

\end{proof}

\subsection{Convergence of local times of a continuous-time simple random walk} 
\label{subsec:appendix_local_times}

The following result is classical and proven in various contexts, see e.g. Proposition 1 in \cite{revesz2006local} for a very strong convergence result in the case of a discrete time random walk, as well as Theorem 1.3 of \cite{croydon2017time-changes} valid in a more general context. In order to use such a convergence result in the formulation tailored to our specific needs, we will give a self-contained proof below, which will also provide moment estimates that will be useful to us.

\begin{lemma}
\label{lem:conv_lt_srw}
    Let $(\xi(t))_{t \geq 0}$ be a one-dimensional, simple random walk with generator given by
    \begin{equation}
\label{generator_srw} 
\mathcal{L} f(x) = f(x+1) - 2 f(x) + f(x-1), \qquad x \in \mathbb{Z}, 
\end{equation}
for all $f: \mathbb{Z} \to \mathbb{R}$. Let, for all $T>0$,
    \[\ell^{\xi,T}_t(x) := \frac{1}{\sqrt{T}} \int_0^{Tt} \mathbf{1}_{\{\xi(s) = \lfloor \sqrt{T} x \rfloor \}} \, ds, \qquad  t \geq 0, \quad x \in \mathbb{R},\]
    and let $(\tilde{\ell}^{T}_t(x))_{x \in \mathbb{R}, t \geq 0}$ be the continuous version, obtained by a linear interpolation in space
    \[\tilde{\ell}^{\xi, T}_t(x) = \left(\lceil \sqrt{T} x \rceil - \sqrt{T} x\right)\ell^{\xi,T}_t(x) + \left(\sqrt{T} x - \lfloor \sqrt{T} x \rfloor \right) \ell^{\xi, T}_t(x+1),  \qquad t \geq 0,
    \quad x \in \mathbb{R},\]
    where $\lfloor x \rfloor$ designates the integer part of $x$, and $\lceil x \rceil := \lfloor x \rfloor + 1$.
    Then, we have the joint convergences in law 
    \[ \left(\frac{1}{\sqrt{T}} \xi(Tt) \right)_{t \geq 0} \Longrightarrow (B_t)_{t \geq 0} \qquad \hbox{in} \quad D  \]
     for the Skorohod topology and 
    \[(\tilde{\ell}^{\xi,T}_t(x))_{x \in \mathbb{R}, t \geq 0} \Longrightarrow  (L^B(t,x))_{x \in \mathbb{R}, t \geq 0} \qquad \hbox{in} \quad C([0,\infty) \times \mathbb{R}) \]
    for the topology of uniform convergence on compact sets. Here $(B_t)_{t \geq 0}$ is a Brownian motion started from $0$ with diffusion coefficient $\sqrt{2}$ and $(L^B(t,x))_{x \in \mathbb{R}, t \geq 0}$ is its family of local times.  
\end{lemma}

\begin{proof}
We start by proving tightness of the linearly interpolated local times $\tilde{\ell}^{T}_t(x)$, $(t,x) \in [0,\infty) \times \mathbb{R}$, in $C([0,\infty) \times \mathbb{R})$. We will actually show that for all $p>1$, and all $M>0$ there exists $C>0$ such that
\begin{equation}
\label{eq.tightness.lvt.space}
E^\omega_0[|\tilde{\ell}^{T}_t(x) - \tilde{\ell}^{T}_t(y)|^{2p}] \leq C \, |x-y|^p,
\end{equation} 
and 
\begin{equation}
\label{eq.tightness.lvt.time}
E^{\omega}_0[|\tilde{\ell}^{T}_t(x) - \tilde{\ell}^{T}_s(x)|^{2p}] \leq C \, |t-s|^p, 
\end{equation} 
for all $t,s \in [0,M]$ and $x,y \in [-M,M]$. Since $\tilde{\ell}^{T}_t(\cdot)$ is the piecewise linear interpolation of $\ell^{T}_t(\cdot)$, by a standard argument, equation \eqref{eq.tightness.lvt.space} will follow upon proving 
\begin{equation}
\label{eq.tightness.lt}
E^\omega_0[|\ell^{T}_t(x) - \ell^{T}_t(y)|^{2p}] \leq C \, |x-y|^p, \qquad t \in [0,M], \quad x,y \in [-M,M] \cap \frac{1}{\sqrt{T}} \mathbb{Z}. 
\end{equation}  
To prove \eqref{eq.tightness.lt}, we use a discrete Tanaka formula. 
Namely, for all $k \in \mathbb{Z}$, we apply Dynkin's formula with the function $f_k : x \mapsto \frac{1}{2} |x-k|$ (see e.g. Proposition B.1 in \cite{Mourrat2012CLT}) : noting that $\mathcal{L} f_k (x) = \mathbf{1}_{\{k\}}(x)$, while the carr\'{e} du champ operator $\Gamma$ applied to $f_k$ yields $\Gamma(f_k,f_k)(x) = 1/2$ for all $x \in \mathbb{Z}$,  we get 
\[\frac{1}{2} |\xi(t)-k| = \frac{1}{2} |k| + \int_0^{t} \mathbf{1}_{\{\xi(s) = k\}} ds + M_{f_k}(t), \qquad t \geq 0,  \]
where $M_{f_k}$ is a martingale with quadratic variation 
\[\langle M_{f_k}, M_{f_k} \rangle_t = \frac{t}{2} \leq t. \]
Therefore, if $x \in \frac{1}{\sqrt{T}} \mathbb{Z}$, we get
\[\tilde{\ell}^{T}_t(x) := \frac{1}{ \sqrt{T}}  \int_0^{Tt} \mathbf{1}_{\{\xi(s) = \sqrt{T} x \}} ds =\frac{1}{2\sqrt{T}}  \left( \left|\xi(Tt) - \sqrt{T}x \right| -   |\sqrt{T} x| \right) - \frac{1}{\sqrt{T}} M_{f_{\sqrt{T}x}}(Tt),  \]
whence, applying the triangle inequality and the inequality $(a+b)^p \leq 2^p (a^p + b^p)$, we get
\[E^\omega_0[\tilde{\ell}^{T}_t(x)^{2p}] \leq 2^p \left( E^\omega_0 \left[ \left| \frac{1}{\sqrt{T}} \xi(Tt) \right|^{2p} \right] + E^\omega_0 \left[ \left|  \frac{1}{\sqrt{T}} M_{f_{\sqrt{T}x}}(Tt) \right|^{2p} \right] \right) \]
Note that $\left(\frac{1}{\sqrt{T}} \xi(Tt) \right)_{t \geq 0}$ and $\left( \frac{1}{\sqrt{T}} M_{f_{\sqrt{T}x}}(Tt) \right)_{t \geq 0}$ are both martingales with quadratic variation bounded by $\frac{1}{T} Tt = t$, so, by the BDG inequality, both expectations above are bounded by $t^p$, up to a constant depending only on $p$.
Therefore we obtain the following bound which is uniform in $x \in \frac{1}{\sqrt{T}} \mathbb{Z}$:
\begin{equation}
\label{eq.bound.lp.lt}
E^\omega_0[(\ell^{T}_t(x))^{2p}] \leq c(p) \, t^p, 
\end{equation}
where $c(p)>0$ depends only on $p$. Let now $k,j \in \mathbb{Z}$ with $k < j$ and consider the test function $f_{k,j} := f_{k} - f_{j}$. Then $\mathcal{L} f_{k,j}(x) =  \left( \mathbf{1}_{\{k\}}(x) - \mathbf{1}_{\{j\}}(x)  \right)$ and 
\[ \Gamma(f_{k,j},f_{k,j})(x) = \begin{cases}
2 \quad &\text{if} \quad k < x <j \\ 
1 \quad  &\text{if} \quad x =k \, \, \text{or} \, \, x=j \\ 
0 \quad &\text{otherwise},
\end{cases} \]
so in particular $\Gamma(f_{k,j},f_{k,j})(x) \leq 2 \mathbf{1}_{\{k \leq x \leq j\}}$. Hence,
\[\frac{1}{2} \left\{|\xi(t)-k|  - |\xi(t)-j|\right\}= \frac{1}{2} (|k|-|j|) + \int_0^{t} \left(\mathbf{1}_{\{\xi(s) = k\}} - \mathbf{1}_{\{\xi(s) = j\}} \right) ds + M_{f_{k,j}}(t), \qquad t \geq 0,  \]
where $(M_{f_{k,j}}(t))_{t \geq 0}$ is a martingale with quadratic variation $\langle M_{f_{k,j}}, M_{f_{k,j}} \rangle_t \leq 2 \int_0^t \mathbf{1}_{\{k \leq \xi(s) \leq j\}} \, ds$. 
We deduce that, for all $x,y \in \frac{1}{\sqrt{T}} \mathbb{Z}$ with $x<y$,
\[\ell^{V,T}_t(x) - \ell^{V,T}_t(y) = \frac{1}{2\sqrt{T}} \left\{ \left( |\xi(T t)-\sqrt{T} x| - |\sqrt{T} x| \right) - \left(|\xi(T t)-\sqrt{T} y| - |\sqrt{T} y| \right)  \right\} - M^T(t) \]
where the martingale $M^T(t) := \frac{1}{\sqrt{T}} M_{f_{\sqrt{T}x,\sqrt{T}y}}(tT)$, $t \geq 0$, has quadratic variation
$$\langle M^T,M^T \rangle_t = \frac{1}{T} \int_{0}^{Tt} \Gamma(f_{\sqrt{T}x,\sqrt{T}y},f_{\sqrt{T}x,\sqrt{T}y})(\xi(s)) \, ds \leq \frac{2}{T} \int_0^{Tt} \mathbf{1}_{\{\sqrt{T} x \leq \xi (s) \leq \sqrt{T} y \}} ds, $$
that is 
\[\langle M^T,M^T \rangle_t  \leq \frac{2}{\sqrt{T}} \sum_{\substack{x \leq z \leq y \\ z \in \frac{1}{\sqrt{T}} \mathbb{Z}}} \tilde{\ell}^{T}_t(z).\]
Thus, using the triangle inequality once more, we get
\[\begin{split}
E^\omega_0[|\tilde{\ell}^{T}_t(x) - \tilde{\ell}^{T}_t(y)|^{2p}] &\leq 2^{2p} \bigg\{ |x-y|^{2p} + E^\omega_0 \bigg[ \bigg(\frac{2}{\sqrt{T}} \sum_{\substack{x \leq z \leq y \\ z \in \frac{1}{\sqrt{T}} \mathbb{Z}}} \tilde{\ell}^{T}_t(z) \bigg)^{p} \bigg] \bigg\}.
\end{split}\]
Note that the sum in the second term of the right-hand side has $\sqrt{T}|y-x| +1$ terms. 
Since $x<y$ and $x,y \in \frac{1}{\sqrt{T}} \mathbb{Z}$, in particular $1 \leq \sqrt{T}|y-x|$. Therefore, using Jensen's inequality as well as the estimate \eqref{eq.bound.lp.lt} with $p$ replaced by $p/2$,  
we deduce that
\begin{equation}
\label{eq:bound.proof.tight}
E^\omega_0[|\tilde{\ell}^{T}_t(x) - \tilde{\ell}^{T}_t(y)|^{2p}]  \leq C(p) \left\{|x-y|^{2p} + t^{p/2} |y-x|^p \right\},
\end{equation}
where $C(p)>0$ depends only on $p$,
and the claimed inequality \eqref{eq.tightness.lt} follows. 
\end{proof}

\subsection{A few topological facts}

We recall the following fact, the proof of which is standard.

\begin{lemma}
\label{lem:conv.inv}
    Let $(\varphi_n)_{n \geq 0}$ a sequence of continuous strictly increasing functions from $[0,\infty)$ onto $[0,\infty)$. We assume that $\varphi_n (t) \underset{n \to \infty}{\longrightarrow} \varphi(t)$, uniformly over compact intervals, where $\varphi$ is a continuous, strictly increasing function. Then, we have the convergence
    \[ \varphi_n^{-1} (u) \underset{n \to \infty}{\longrightarrow} \varphi^{-1}(u) \]
    locally uniformly over $u \in [0,\infty)$.
    \end{lemma}


For the next lemma, inspired by \cite{billingsleySndEdition}, we denote by $D_0$ the set of non-decreasing c\`{a}dl\`{a}g functions from $[0,\infty)$ onto $[0,\infty)$. We recall the following fact, proven in page 151 of \cite{billingsleySndEdition}.

\begin{lemma}
\label{lem:cont_comp}
Let $\Psi : D \times D_0 \to D$ the composition map 
$$\Psi(x,\varphi) := x \circ \varphi, \qquad (x,\varphi) \in D \times D_0. $$
Then $\Psi$ is continuous at every pair $(x,\varphi) \in D \times D_0$ such that $x \in C$. 
\end{lemma}
 
As a result follows the following fact
 
\begin{lemma}
\label{lem:composition}
Let $(X^{(T)},\varphi^{(T)})_{T >0}$ be a family of random variables with values in $D \times D_0$, and converging in law to some limit 
$(X,\varphi)$ in $D \times D_0$, as $T \to \infty$. We assume $X $ has continuous paths a.s. Then the random variable $X^{(T)} \circ \varphi^{(T)}$ converges in law to $X \circ \varphi$ as $T \to \infty$, for the Sokorohod topology on $D$. 
\end{lemma} 

\begin{remark}
We can generalize the above statement to two-dimensional processes. More precisely, by Lemma \ref{lem:cont_comp}  we see that the map $\Psi^{\otimes 2} : (D \times D_0)^2 \to D \times D$ defined by
$$ \Psi^{\otimes 2}((x_1,\varphi_1),(x_1,\varphi_2)) = (x_1 \circ \varphi_1, x_2 \circ \varphi_2), \qquad ((x_1,\varphi_1),(x_2,\varphi_2)) \in (D \times D_0)^2  $$
is continuous at every $(x_1,\varphi_1),(x_2,\varphi_2)$ such that $x_1,x_2 \in C$. Hence if 
$$((X_1^{(T)},\varphi_1^{(T)}),(X_2^{(T)},\varphi_2^{(T)})), \qquad T>0,$$
is a family of random variables with values in $(D \times D_0)^2$, and converging in law to some limit $((X_1,\varphi_1), (X_2,\varphi_2))$ in $(D \times D_0)^2$, as $T \to \infty$, and $X_1$ and $X_2$ both have continuous paths a.s., then the random variable $(X_1^{(T)} \circ \varphi_1^{(T)},X_2^{(T)} \circ \varphi_2^{(T)})$ converges in law to $(X_1 \circ \varphi_1, X_2 \circ \varphi_2)$ as $T \to \infty$, on $D \times D$ (where $D$ is equipped with the Sokorohod topology). 

\end{remark}

From the above lemma, using the subsequential characterization of tightness, we at once obtain the following result:

\begin{corollary}
\label{cor:comp_tightness}
Let $(X^{(T)},\varphi^{(T)})_{T >0}$ be a family of random variables with values in $D \times D_0$. We assume that the sequence of random variables $(X^{(T)})_{T>0}$ is tight in $D$, that $(\varphi^{(T)})_{T >0}$ is tight in $D_0$, and that each subsequential limit of $(X^{(T)})_{T>0}$ has continuous paths a.s. Then the sequence of random variables $(X^{(T)}\circ\varphi^{(T)})_{T>0}$ is tight in $D$. 
\end{corollary}

\bibliographystyle{amsplain}
\bibliography{line_model_in_degenerate_environment_to_submit}

\providecommand{\bysame}{\leavevmode\hbox to3em{\hrulefill}\thinspace}
\providecommand{\MR}{\relax\ifhmode\unskip\space\fi MR }
\providecommand{\MRhref}[2]{%
  \href{http://www.ams.org/mathscinet-getitem?mr=#1}{#2}
}
\providecommand{\href}[2]{#2}
\begin{thebibliography}{10}

\bibitem{AHLE2022109306}
T.~D. Ahle, \emph{Sharp and simple bounds for the raw moments of the binomial
  and {P}oisson distributions}, Statistics $\&$ Probability Letters
  \textbf{182} (2022), 109306.

\bibitem{ABDH}
S.~Andres, M.~T. Barlow, J.~D. Deuschel, and B.~M. Hambly, \emph{Invariance
  principle for the random conductance model}, Probability Theory and Related
  Fields \textbf{156} (2013), no.~3, 535--580.

\bibitem{Andres2013HarnackIO}
S.~Andres, J.-D. Deuschel, and M.~Slowik, \emph{Harnack inequalities on
  weighted graphs and some applications to the random conductance model},
  Probability Theory and Related Fields \textbf{164} (2013), 931--977.

\bibitem{Andres2015invariance}
\bysame, \emph{Invariance principle for the random conductance model in a
  degenerate ergodic environment}, The Annals of Probability \textbf{43}
  (2015), no.~4, 1866--1891.

\bibitem{Andres:2016aa}
\bysame, \emph{Heat kernel estimates for random walks with degenerate weights},
  Electronic Journal of Probability \textbf{21} (2016), no.~none, 1--21.

\bibitem{benarous2005}
G.~Ben Arous and J.~Cern{\'y}, \emph{Bouchaud's model exhibits two different
  aging regimes in dimension one}, The Annals of Applied Probability
  \textbf{15} (2005), no.~2, 1161 -- 1192.

\bibitem{Barlow:2010aa}
M.~T. Barlow and J.~D. Deuschel, \emph{Invariance principle for the random
  conductance model with unbounded conductances}, The Annals of Probability
  \textbf{38} (2010), no.~1, 234--276.

\bibitem{bertoin1996levy}
J.~Bertoin, \emph{L{\'e}vy processes}, vol. 121, Cambridge university press
  Cambridge, 1996.

\bibitem{Bertoin1999}
\bysame, \emph{Subordinators: Examples and applications}, pp.~1--91, Springer
  Berlin Heidelberg, Berlin, Heidelberg, 1999.

\bibitem{billingsleySndEdition}
P.~Billingsley, \emph{Convergence of probability measures}, Wiley Series in
  Prob- ability and Statistics: Probability and Statistics. John Wiley \& Sons
  Inc., New York, second edition, 1999.

\bibitem{BISKUP}
M.~Biskup, \emph{{Recent progress on the Random Conductance Model}},
  Probability Surveys \textbf{8} (2011), no.~none, 294 -- 373.

\bibitem{CF}
Z.-Q. Chen and M.~Fukushima, \emph{Symmetric {Markov} processes, time change,
  and boundary theory}, Lond. Math. Soc. Monogr. Ser., vol.~35, Princeton, NJ:
  Princeton University Press, 2012 (English).

\bibitem{croydon2017time-changes}
D.~Croydon, B.~Hambly, and T.~Kumagai, \emph{{Time-changes of stochastic
  processes associated with resistance forms}}, Electronic Journal of
  Probability \textbf{22} (2017), no.~none, 1 -- 41.

\bibitem{DF1}
J.-D. Deuschel and R.~Fukushima, \emph{Quenched tail estimate for the random
  walk in random scenery and in random layered conductance}, Stochastic
  Processes and their Applications (2016).

\bibitem{DF2}
\bysame, \emph{{Quenched tail estimate for the random walk in random scenery
  and in random layered conductance II}}, Electronic Journal of Probability
  \textbf{25} (2020), no.~none, 1 -- 28.

\bibitem{deuschel2024quenchedinvarianceprinciplerandom}
J.-D. Deuschel, M.~Slowik, and W.~Weng, \emph{Quenched invariance principle for
  random walks in random environments admitting a cycle decomposition}, 2024.

\bibitem{FIN}
L.~R.~G. Fontes, M.~Isopi, and C.~M. Newman, \emph{{Random walks with strongly
  inhomogeneous rates and singular diffusions: convergence, localization and
  aging in one dimension}}, The Annals of Probability \textbf{30} (2002),
  no.~2, 579 -- 604.

\bibitem{FO}
M.~Fukushima and Y.~Oshima, \emph{On the skew product of symmetric diffusion
  processes}, Forum Math. \textbf{1} (1989), no.~2, 103--142 (English).

\bibitem{kesten1979limit}
H.~Kesten and F.~Spitzer, \emph{A limit theorem related to a new class of self
  similar processes}, Zeitschrift f\"{u}r Wahrscheinlichkeitstheorie und
  verwandte Gebiete \textbf{50} (1979), no.~1, 5--25.

\bibitem{komorowski2012fluctuations}
Tomasz Komorowski, Claudio Landim, and Stefano Olla, \emph{Fluctuations in
  {Markov} processes. {Time} symmetry and martingale approximation.},
  Grundlehren Math. Wiss., vol. 345, Berlin: Springer, 2012 (English).

\bibitem{lawler2010random}
G.F. Lawler and V.~Limic, \emph{Random walk: A modern introduction}, Cambridge
  Studies in Advanced Mathematics, Cambridge University Press, 2010.

\bibitem{Liggett2010ContinuousTM}
T.~M. Liggett, \emph{Continuous {T}ime {M}arkov {P}rocesses: {A}n
  {I}ntroduction}, vol. 113, Graduate Studies in Mathematics, American
  Mathematical Soc., 2010.

\bibitem{Mourrat2012CLT}
J.-C. Mourrat, \emph{A quantitative central limit theorem for the random walk
  among random conductances}, Electron. J. Probab. \textbf{17} (2012), no. 97,
  1--17.

\bibitem{revesz2006local}
P.~R{\'e}v{\'e}sz, \emph{Local time and invariance}, Analytical Methods in
  Probability Theory: Proceedings of the Conference Held at Oberwolfach,
  Germany, June 9--14, 1980, Springer, 2006, pp.~128--145.

\bibitem{revuz2013continuous}
D.~Revuz and M.~Yor, \emph{Continuous martingales and {B}rownian motion}, vol.
  293, Springer Science \& Business Media, 2013.

\bibitem{Zeitouni:2004aa}
O.~Zeitouni, \emph{Part ii: Random walks in random environment}, pp.~190--312,
  Springer Berlin Heidelberg, Berlin, Heidelberg, 2004.

\end{thebibliography}

\end{document}